%% file: main.tex
\documentclass[prb]{revtex4-2}

\usepackage{color}

\usepackage{hyperref}
\usepackage{float}
\usepackage{comment}
\definecolor{eggplant}{RGB}{126,93,181}
\definecolor{cayenne}{RGB}{148,17,0}
\definecolor{teal}{RGB}{0,145,147}
\definecolor{blueberry}{RGB}{4,51,255}

\usepackage{graphicx}

\usepackage{amsmath}     
\usepackage{amsthm}
  \usepackage{cleveref}	
  	\crefname{figure}{Fig.}{Figs.}
  	\crefname{table}{Table}{Tables}
  	\crefname{equation}{Eq.}{Eqs.}
  	\crefname{section}{Section}{Sections}
  	\crefname{subsection}{Section}{Sections}
  	\crefname{subsubsection}{Section}{Sections}
  	\newcounter{algorithm}
\crefname{algorithm}{Algorithm}{Algorithms}
\Crefname{algorithm}{Algorithm}{Algorithms}

  \usepackage{amssymb}
  \usepackage{bm}
  \usepackage{mathrsfs}
  \usepackage{titlesec}
  \usepackage{subcaption}
\usepackage{mdframed}
  \usepackage[noend]{algpseudocode}
\newtheorem{definition}{Definition}[section]
\newtheorem{proposition}[definition]{Proposition}

\newtheorem{theorem}[definition]{Theorem}
\newtheorem{lemma}[definition]{Lemma}

\date{\today}

\begin{document}

\author{Israa Fakih$^{*}$}
\author{Laura Grigori$^{*}$}
\author{Karl Pierce$^{\dagger}$}
\affiliation{$^{*}$ PSI Center for Scientific Computing, Theory and Data, Villigen PSI; Institute of Mathematics, EPFL, Switzerland}
\affiliation{$^{\dagger}$ The University of Maryland, College Park, College Park MD 20742 USA}
\title{Accelerating the Canonical Polyadic Alternating Least Squares Optimization via a Randomized Interpolative Decomposition}

\maketitle

\section*{Abstract}
We present a novel leverage score-based sampling strategy for the randomized alternating least squares optimization (ALS) of the canonical polyadic decomposition (CPD-ALS).
Unlike previous strategies, we  determine row-wise samples for the CPD-ALS problem from the leverage scores of the tensor which is being decomposed (denoted the target tensor).
We demonstrate that, when rows are sampled according to the leverage score distribution of the matricized target tensor, each least squares subproblem of the CPD-ALS problem achieves $(1+\epsilon)-$relative accuracy in the residual norm with probability at least $1-\delta$ using a sampling $s=\frac{R\gamma}{\beta} \max\left(\frac{4}{\delta \epsilon}, \frac{144\ln(2R/\delta)}{\epsilon_{0}^{2}}\right)$, where $\epsilon_{0}$ is a constant, $\beta$ is leverage score's approximation constant, $R$ is the target rank and $\gamma$ captures the coherence between the Khatri–Rao product (KRP) of the CPD factor matrices and the exact KRP; $\gamma$ decreases as the ALS iterates converge. 
To efficiently approximate the leverage score distribution for each matricization of the target tensor without explicitly computing leverage scores we employ a strong rank-revealing QR (sRRQR) factorizations.
By construction, this QR-based leverage score sampling method outperforms previously published schemes as it does not, in principle, require the resampling of the target tensor or recomputing the leverage scores of the KRP, minimizing the computational and storage overhead of the CPD-ALS procedure.
Furthermore, because the approach does not depend on computing the leverage scores of individual factor matrices, the method is still effective when the CP rank is greater than the dimension of any one mode of the target tensor. 
To address the high computational cost of the CP-sRRQR on a higher-order matricized tensor, we use a recently developed sparse embedding-based sRRQR strategy called SE-QRCS.
We demonstrate the effectiveness of this novel approach on a collection of synthetic and real tensors and compare the results to the current state-of-the-art sampled CPD-ALS strategy.

\input{Introduction}
\input{Preliminaries}
\input{Proposed_method}
\input{Theoretical_Results}
\input{Computational_Experiments}
\input{Conclusion}

\section{Acknowledgments}
    This work was supported by the University of Maryland, College Park. 
    
\bibliography{kmprefs,bibliography}

\end{document}

%% file: Introduction.tex
\section{Introduction}
Tensors (multi-dimensional arrays) are essential computational structures necessary in applications such as computational modeling and multi-way data analysis. 
It is well understood that tensor-based algorithms are plagued by the ``curse of dimensionality". 
To address the curse of dimensionality, decompositions in tensor formats such as the canonical polyadic decomposition (CPD) can be used \cite{hitchcock1927expression,carroll1970analysis,harshman1970foundations}.
Unfortunately, the algorithms used to optimize the CPD are, too, plagued by the curse of dimensionality.
This is especially prevalent in the popular alternating least-squares CPD (CPD-ALS) \cite{VRG:kroonenberg:1980:P,VRG:beylkin:2002:PNAS} optimization strategy, which iteratively solves a set of highly overdetermined least-squares problems.
Here, we present our research into the application of randomized and interpolative methods to reduce the computational and storage complexity of the CPD-ALS optimization procedure.

The CPD represents an order-$N$ tensor as a sum of $R$ rank-one component tensors.
For example, the CPD of the target tensor $\mathcal{T}\in \mathbb{R}^{n_1 \times n_2 \times n_3}$ is 
\begin{equation}
    \label{eq:CPD_ex}
    \mathcal{T} = [\![ \mathbf{F}_1^{*}, \mathbf{F}_2^{*},\mathbf{F}_3^{*} ] \!], 
\end{equation}
where $\{\mathbf{F}_1^{*}, \mathbf{F}_2^{*}, \mathbf{F}_3^{*}\}$ denote the CP factor matrices such that $\mathbf{F}^{*}_i \in \mathbb{R}^{n_{i} \times R_\mathrm{CP}}$, where $n_{i}$ is the dimension of the $i$th mode of $\mathcal{T}$ and $R_\mathrm{CP}$ is the CP rank. 
The CP rank is defined as the fewest number of rank-one components which satisfies the equality in \eqref{eq:CPD_ex}.
Currently, there is no straightforward algorithm to determine the CP rank of a tensor.
Therefore, the goal of CPD optimization is to approximate the tensor $\mathcal{T}$
using $R$ rank-one component tensors.

The CPD-ALS of an order-$N$ tensor is formulated by splitting the non-linear CPD optimization problem into $N$, linear least-squares (LS) problems, one for each CP factor matrix.
These $N$ LS problems are solved successively until some global convergence criterion is achieved.
For example, given an arbitrary set of rank-$R$ factor matrices $\{{\bf F}_1, {\bf F}_2, \dots {\bf F}_N\}$ and an order-$N$ target tensor $\mathcal{T}\in \mathbb{R}^{n_{1}\times n_{2}\times \dots \times n_{N}}$, the LS optimization of the $k$-th mode can be formulated as follows:
\begin{align}\label{eq:ls_orig}
    \min_{\mathbf{F}_{k}} \| 
    \mathbf{W}_{k}(\mathbf{F}_{k})^T - \mathbf{T}_{k}^T \|^2_F,
\end{align}
where $\mathbf{W}_{k} = \mathbf{F}_{N} \odot \mathbf{F}_{N-1} \odot \dots \mathbf{F}_{k+1} \odot \mathbf{F}_{k-1} \odot \dots \mathbf{F}_{1} \in \mathbb{R}^{m\times R}$ is the Khatri-Rao product (KRP) of all factors but $\mathbf{F}_{k}\in \mathbb{R}^{n_{k}\times R}$ and $\mathbf{T}_{k} \in \mathbb{R}^{n_{k}\times m}$ is the $k-$th mode matricization of the target tensor $\mathcal{T}$ with $m=\prod_{i\neq k}^{N}n_{i}$.
Because, in general, $n_{k}\ll m$, the LS problem \eqref{eq:ls_orig} is tall and skinny. 
As this is true for every LS subproblem, acceleration can be achieved by sketching each LS problem with a sampling matrix $\mathbf{S}_k\in \mathbb{R}^{s\times m}$ and solving each sketched LS problem, that is,
\[
    \min_{\mathbf{F}_k} \| 
    \mathbf{S}_k \mathbf{W}_{k} (\mathbf{F}_{k})^{T} - \mathbf{S}_k\mathbf{T}_{k}^{T} \|^2_F.
\]
Recently, Larsen et al.~\cite{larsen2022practical} and Battaglino et al.~\cite{battaglino2018practical} introduced a randomized CPD-ALS algorithm that defines $\bf{S}_k$ as a row-wise sampling matrix.
In their work, non-zero entries of $\bf{S}_k$ are determined by taking the row-wise leverage scores of each KRP as a sampling distribution.
Because forming and computing the leverage scores of the full KRP is prohibitively expensive, Battaglino et. al. take advantage of the tensor product structure of the KRP to approximate the leverage scores as a Kronecker product of the leverage scores of each component factor matrix.
Therefore, samples of the KRP can be efficiently determined by sampling the probability distribution of each component factor matrix independently.
More recently, efficient KRP-based sampling schemes have been further refined by Bharadwaj et al.~\cite{Bharadwaj:2023:NURIPS}.

Although this KRP-based method is an efficient sampling scheme, the approach has several downsides. 
First, the KRP of each LS subproblem changes after any single LS update; therefore, the leverage scores for each factor matrix must be recomputed many times during the ALS optimization procedure.
Second, because the leverage score distribution changes throughout the optimization, the components of each LS subproblem (both target tensor and KRP) must be resampled every LS update.
As one might expect, in the limit of large or complicated target tensors,
resampling becomes the bottleneck of an efficient sampled CPD-ALS optimization.
With this leverage score-based sampling method, the target tensor $\mathcal{T}$ is sampled $\mu N$ times, where $\mu$ is the number of ALS iterations, and $N$ is the order of $\mathcal{T}$.
Finally, $R_{CP}\gg n_{k}$, for some $k$,  the leverage score-based probability distribution for the factor matrix $\mathbf{F}\in \mathbb{R}^{n_{k}\times R}$ becomes a uniform distribution, making KRP-based sampling ineffective \cite{drineas2012fast}. 

To address these issues, in this work, we introduce an improved strategy to determine the sampling matrices for CPD-ALS.
With this strategy, our objective is to efficiently compute a single sampling matrix for each LS subproblem which is fixed for the entire ALS procedure; thus reducing the number of times $\mathcal{T}$ must be sampled from $\mu N$ to $N$.
Furthermore, because we do not require $\mathcal{T}$ to be resampled during the ALS procedure, we are free to remove the tensor from local memory and develop an efficient matrix-free CPD-ALS algorithm. 
To achieve this goal, we attempt to identify the positions of the most strongly correlated rows of the best, formally-unknown, KRP associated with the $k$-th mode's LS problem using the column-wise pivoted QR
\[
    \mathbf{T}_{k} \mathbf{\Pi}_{k} = \mathbf{Q}_{k} \mathbf{R}_{k}.
\]
In principle, this QR carries the same cost as a single standard LS updated.
However, practically, we may recognize that this QR factorization, much like the LS problem, is significantly overdetermined, and, therefore, we may leverage randomized linear algebra strategies.
In the following we introduce our (randomized) pivoted QR based leverage score sampling scheme and prove in Theorem \ref{main theorem} that one can efficiently approximate the ALS loss functions using samples determined from the fixed, matricized target tensor in each LS subproblem. The number of samples depends on a parameter $\gamma$ that quantifies the discrepancy between the leverage scores of the KRP and those of the matricized tensor $\mathbf{T}_{k}$.


%% file: Preliminaries.tex
\section{Preliminaries}
In this section, we provide pertinent background for tensor operations, CPD factorization, and alternating least squares algorithm, leverage scores and the sparse embedding-based rank revealing QR (SE-QRCS).
Throughout this paper, tensors are denoted by bold calligraphic letters (e.g., $\mathcal{A}$) matrices by bold capital letters (e.g., $\mathbf{A}$) vectors by bold lowercase letters (e.g., $\mathbf{v}$) and scalars by lowercase letters ( e.g., $a$). 
Given a matrix $\mathbf{A}\in \mathbb{R}^{n\times m}$, we denote the $i^{th}$ row of $\bf{A}$ as $\mathbf{a}_{i,:}$ and the $j^{th}$ column of $\bf{A}$ as $\mathbf{a}_{:,j}$.

\subsection{Tensor Operations and Contractions}
A tensor $\mathcal{T}\in \mathbb{R}^{n_{1}\times n_{2}\times\dots \times n_{N}}$ is a multidimensional array of order $N$, and each $n_{k}$ denotes the size of the tensor along the mode $k$. The mode-$k$ unfolding or matricization of  $\mathcal{T}$ rearranges the higher-order tensor into a matrix $\mathbf{T}_{k}\in \mathbb{R}^{n_{k}\times m}$, where $m=\prod_{i\neq k}n_{i}$, such that 
$\mathbf{T}_{k}(i_{k},j)=\mathcal{T}(i_{1},i_{2},\dots,i_{N})$ with 
\[j=1+\sum_{n\neq k}(i_{n}-1)J_{n} \quad \text{where} \quad J_{n}=\prod_{\substack{d=1\\ d \neq k}}^{n-1} n_d.\]
Given two matrices $\mathbf{A}\in \mathbb{R}^{n\times m}$ and $\textbf{B}\in \mathbb{R}^{l\times k}$, the Kronecker product $\mathbf{A}\otimes \mathbf{B}\in \mathbb{R}^{nl\times mk}$ is defined as
\[\mathbf{A}\otimes \mathbf{B} = \begin{bmatrix}
    a_{11}\textbf{B} & a_{12}\textbf{B} & \dots &a_{1m}\textbf{B}\\
    \vdots & \vdots & \ddots & \vdots\\
    a_{n1}\textbf{B} & a_{n2}\textbf{B} & \dots & a_{nm}\textbf{B}
\end{bmatrix}.\]
In the case where $m=k$, the Khatri- Rao product $\mathbf{A}\odot \mathbf{B}\in \mathbb{R}^{nl\times m}$ is defined as the column-wise Kronecker-product of the two matrices
\[\mathbf{A} \odot \mathbf{B} = \begin{bmatrix}
    \mathbf{a}_{:,1}\otimes \mathbf{b}_{:,1} & \mathbf{a}_{:,2}\otimes \mathbf{b}_{:,2}  & \dots & \mathbf{a}_{:,m}\otimes \mathbf{b}_{:,m}
\end{bmatrix}.\]
Finally, in the case where $m=k$ and $n=l$, the Hadamard product, $\textbf{A}\circledast \mathbf{B}\in \mathbb{R}^{l\times m}$ is the elemental multiplication of the entries  of $\textbf{A}$ and $\textbf{B}$.

The canonical polyadic decomposition (CPD) aims to approximate an order-$N$ target tensor $\mathcal{T}$ as the sum of $R$ rank$-1$ tensors \cite{hitchcock1927expression,carroll1970analysis,harshman1970foundations,kolda2009tensor}]
\begin{equation}
\label{CPD-equation}
    \mathcal{T} \approx \mathcal{\tilde{T}}= \sum_{r=1}^{R}\mathbf{f}^{r}_{1}\circ \mathbf{f}^{r}_{2}\dots \circ \mathbf{f}^{r}_{N},
\end{equation}
where $\mathbf{f}^{r}_{k}\in \mathbb{R}^{n_{k}}$ and $\circ$ is the outer vector product. 
By defining the CPD factor matrices as $\mathbf{F_{k}}=\begin{bmatrix}
    \mathbf{f}^{1}_{k} &  \mathbf{f}^{2}_{k} \dots  \mathbf{f}^{R}_{k}
\end{bmatrix}\in \mathbb{R}^{n_{k}\times R}$, the decomposition can be written component-wise as 
\[
\tilde{t}_{i_1,i_2,\dots,i_N} =\sum_{r=1}^{R} \lambda_r \, \mathbf{F}^{i_1 r}_1 \, \mathbf{F}^{i_2 r}_2 \dots \mathbf{F}^{i_N r}_N,
\]
where $\mathbf{F}^{i_{k}r}_{k}$ is the $(i_{k},r)$ entry of $\mathbf{F}_{k}$ and $\lambda$ enforces column-wise normalization of the factors.
Using the CPD factor matrix representation, one can express the mode-$k$ unfolding of $\tilde{\mathcal{T}}$ as
\begin{equation*}
    \mathbf{\tilde{T}}_{k} = \mathbf{F}_{k}\mathbf{W}_{k}^{T} \quad \text{ where} \quad \mathbf{W}_{k}=\mathbf{F}_{1}\odot\mathbf{F}_{2}\odot\dots\odot\mathbf{F}_{k-1}\odot\mathbf{F}_{k+1}\odot\dots\mathbf{F}_{N}
\end{equation*}
The CP rank is defined as the smallest number of rank-1 components such that \eqref{CPD-equation} is an equality.
We make the distinction that, because there is no closed-form equation to determine the CP rank, the value $R$ is simply called the rank of the approximation and not the CP rank.
\subsection{Alternating Least Squares Algorithm}
One of the standard methods for optimizing a collection of CPD factor matrices is through a set of linear LS problems.
These LS problems are formulated by assuming that all but one factor matrix, $\mathbf{F}_{k}$, are fixed.
$\mathbf{F}_{k}$ is then obtained by solving
\begin{equation}
    \label{ALS_update}
    \mathbf{F}_{k} = \arg\min_{\mathbf{F}_{k}}\left\|\mathbf{W}_{k}\mathbf{F}_{k}^{T}-\mathbf{T}_{k}^{T}\right\|_F^2.
\end{equation}
This LS problem can be solved using a QR factorization to form the pseudoinverse of $\mathbf{W}_{k}$, 
\[
\mathbf{F}_k^T = [\mathbf{W}_k]^{\dagger} \mathbf{T}_k^T.
\]
However, this approach is typically impractical due to the cost associated with forming $\mathbf{W}_k$ and  computing its pseudoinverse. Practically, a better method is to form the normal equation 
\[
\mathbf{W}_k^{T} \mathbf{W}_k \mathbf{F}_k^T = \mathbf{W}_k^{T} \mathbf{T}^{T}_k.
\]
The KRP structure can then be utilized to efficiently form the Grammian ($\mathbf{W}_k^T\mathbf{W}_k$) and the matricized tensor times KRP (MTtKRP) ($\mathbf{W}_k^T \mathbf{T}^T_k$).
With the normal equations, one has the advantage of inverting the square, symmetric Grammian matrix, however, this method also squares the condition number of the problem.

The collection of LS problems is iteratively solved until a convergence criterion is satisfied.
One standard way to assess the CPD accuracy is via the fit metric between the target tensor $\mathcal{T}$ and the approximated tensor $\mathcal{\tilde{T}}$, defined as 
\[
\mathrm{Fit}(\mathcal{T}, \mathcal{\tilde{T}}) = 1 - \frac{\|\mathcal{T} - \mathcal{\tilde{T}}\|_F}{\|\mathcal{T}\|_F}
\]
where $\|\mathcal{T}\|_F = \sqrt(\sum^{n_1,n_2,\dots n_K}_{i,j,...,n} t_{i,j,...,n}^2)$. The fit function has an expected range between $[0,1]$, values outside this range suggest a failure in the ALS optimization procedure.

\subsection{Leverage Scores and Coherence}
For a given matrix $\mathbf{W}\in \mathbb{R}^{m\times n}$ with $m\gg n$, let $\mathbf{U}\in \mathbb{R}^{m \times n}$ be an orthogonal basis for the column space of $\mathbf{W}$. The row leverage score of $\mathbf{W}$ \cite{drineas2012fast} gives the importance of each row and is defined by 
\begin{equation*}
    l_{i}(\mathbf{W}) = \|\mathbf{u}_{i,:}\|_{2}^{2} \quad \text{for all } \quad i\in \{1,\dots,m \}.
\end{equation*}
The coherence of a matrix $\mathbf{W}$ is defined as its maximum leverage score. A high coherence indicates that a certain number of rows are highly important in the sense that they contribute significantly more than the others. In this case, the matrix is said to be coherent. However, when the coherence is low, we say that the matrix is incoherent.
\subsection{SE-QRCS Factorization}
Given a matrix $\mathbf{A}\in \mathbb{R}^{n\times m}$ with $n\ll m$ and $1\leq k\leq n$, its partial QR factorization with column pivoting is 
\[
\mathbf{A}\mathbf{\Pi} = \mathbf{Q}\begin{pmatrix}
\mathbf{R}_{11} & \mathbf{R}_{12}\\
 & \mathbf{R}_{22}
\end{pmatrix}
\]
where $\mathbf{\Pi}\in\mathbb{R}^{m\times m}$ is the permutation matrix and $\mathbf{R}_{11}\in \mathbb{R}^{k\times k}$. This factorization is said to satisfy the strong rank-revealing QR properties (sRRQR) \cite{gu1996efficient} if the permutation matrix $\mathbf{\Pi}$ rearranges the columns of $\mathbf{A}$ such that $\mathbf{R}_{11}$ is well conditioned and its singular values approximate the leading $k$ singular values of $\mathbf{A}$, $\mathbf{R}_{22}$ has small norm and its singular values approximate the remaining $n-k$ singular values of $\mathbf{A}$ and the entries of $|\mathbf{R}_{11}^{-1}\mathbf{R}_{12}|$ are bounded.

The number of floating point operations required to compute this factorization is $O(mn^{2})$, which becomes extremely expensive when $m$ is very large. To address this problem, randomization techniques have been used to first embed the high-dimensional space into a lower-dimensional one while approximately preserving its geometry, such as vector norms and scalar products of vectors, up to a distortion $\epsilon$ and with high probability. We first recall the definition of the $\epsilon-$subspace embedding property and present sparse embeddings as an example.
\begin{definition}[$\epsilon-$subspace embedding \cite{woodruff2014sketching}]
Let $\epsilon\in (0,1)$. A sketching matrix $\mathbf{\Omega}\in \mathbb{R}^{l\times m}$ is an $\epsilon-$subspace embedding for a vector space $W\subset \mathbb{R}^{m}$ if with high probability, for any two vectors, $\mathbf{x},\mathbf{y}\in W$
\[|\langle x,y\rangle-\langle\mathbf{\Omega} x,\mathbf{\Omega}y\rangle|\leq \epsilon\|\mathbf{x}\|\|\mathbf{y}\|.\]
\end{definition}
The sketching matrix $\mathbf{\Omega}\in \mathbb{R}^{l\times m}$ is said to be an $(\epsilon,\delta,d)-$Oblivious subspace embedding (OSE) if it satisfies the $\epsilon-$subspace embedding property for any fixed $d-$ dimensional subspace $W\subseteq\mathbb{R}^{m}$ with probability at least $1-\delta$. Sparse subspace embeddings are sparse sketching matrices that satisfy the $(\epsilon,\delta,d)-$OSE property. An example of such embedding is the oblivious sparse norm-approximating projection (OSNAP), introduced in \cite{nelson2013osnap}. One way to construct an OSNAP is as follows: for a fixed sparsity parameter $\tau\geq 1$, for each column $\mathbf{\Omega}_{:,j}$, choose $\tau$ nonzero entries uniformly at random and assign each a value from $\{\frac{-1}{\sqrt{\tau}},\frac{1}{\sqrt{\tau}}\}$, also uniformly at random. 

Although there exist a number of randomized QRCP algorithms \cite{martinsson2017householder},\cite{duersch2017randomized},\cite{xiao2017fast},\cite{grigori2025randomized}, we use the SE-QRCS algorithm, introduced in \cite{fakih2025efficient}, as it is specifically designed for matrices with a large number of columns. 
Let $\mathbf{A}\in \mathbb{R}^{n\times m}$, $n\ll m$ and $\mathbf{\Omega}\in \mathbb{R}^{l\times m}$, where $l$ is the embedding dimension. Assume that $\mathbf{\Omega}$ is a $\tau$-sparse embedding for $\mathrm{range}(\mathbf{A}^{T})$. The SE-QRCS algorithm consists of the following four steps: 
\begin{itemize}
    \item \textbf{Step1 (Sketching) }The  matrix $A$ is sketched on the right by $\Omega^{T}$ to get $\mathbf{A}_{sk}=\mathbf{A}\mathbf{\Omega}^{T}$.
    \item \textbf{Step2 (Pivot selection from sketch) } A strong rank revealing QR is applied to the sketched matrix $\mathbf{A}_{sk}$ 
\[\mathbf{A}_{sk}\mathbf{\Pi}_{sk}=\mathbf{Q}\begin{pmatrix}
    \mathbf{R}_{11}^{sk} & \mathbf{R}_{12}^{sk}\\
     & \mathbf{R}_{22}^{sk}
\end{pmatrix}\]
with $\mathbb{R}_{11}^{sk}\in \mathbb{R}^{k'\times k'}$, where $k'\geq k$.
Interestingly, the method still works well in practice if $k'$ is slightly smaller than $k$.
\item  \textbf{Step3 (Mapping pivots back to $\mathbf{A}$)} The pivot indices selected from $\mathbf{A}_{sk}$ are mapped back to the original matrix $\mathbf{A}$ to  form
\[\mathbf{A}_{sub}=\begin{bmatrix}
    \mathbf{A}_{sub}^{1} & \mathbf{A}_{sub}^{2}
\end{bmatrix}\]
where $\mathbf{A}_{sub}^{1}\in \mathbb{R}^{n\times p}$, the reduced column set consists of the columns of $\mathbf{A}$ corresponding to the selected pivots, and $\mathbf{A}_{sub}^{2}$ contains the remaining columns.
\item \textbf{ Step4 (Final pivot selection)} A strong rank revealing QR is applied to $\mathbf{A}_{sub}^{1}$ to obtain the final pivots $\mathbf{\Pi}$.
\end{itemize}
The expected value of the number of columns $p$ in the reduced column subset matrix $\mathbf{A}_{sub}^{1}$ is shown in \cite{fakih2025efficient} to be 
$$\mathbb{E}(p)\approx m\left[1-\left(1-\frac{k}{l}\right)^{\tau}\right].$$

As the size of the least squares problem \eqref{ALS_update} increases with the size and order of the  tensor $\mathcal{T}$, previous work has focused on reducing computational cost through randomization. In particular, the problem is sketched by sampling certain rows from the Khatri-Rao product $\mathbf{W}_{k}$ and the matricized tensor $\mathbf{T}_{k}^{T}$ as in \cite{battaglino2018practical}, where the sampled rows are chosen uniformly at random. However, since the number of required random samples depends on the coherence of the matrix, this approach works well when the factor matrices are incoherent but may not converge when they are coherent or require a large number of sampled rows.  This can be solved by applying a fast Fourier transform to mix the rows. However, this can lead to 
complex-valued factors even when the tensor has real entries. To mitigate this, a pre-processing step (mixing) 
and a post-processing step (unmixing) are applied, which becomes expensive. Another approach proposed by Larsen et al. \cite{larsen2022practical}, samples rows based on the leverage scores of the Khatri-Rao product $\mathbf{W}_{k}$. As the leverage scores of this product can be upper bounded by the product of the leverage scores of the individual factor matrices \cite{cheng2016spals, battaglino2018practical}, the method estimates and samples according to these factor-level leverage scores at each iteration. Although the algorithm performs well in practice, it has an important limitation: the sampled rows should be computed at the beginning of each iteration which is computationally expensive. Also it treats factor matrices independently and, therefore, cannot capture cases where certain factor matrices become significantly more important when considered jointly. Moreover, the method requires the CP rank to be smaller than the dimension of the factor matrices; otherwise, its leverage score estimates become ineffective, leading to random uniform sampling.

%% file: Proposed_method.tex
\section{CP-sRRQR Method}

In this work, unlike previous approaches, we sample $s$ rows from each LS subproblem \eqref{ALS_update} of the CPD-ALS by applying a randomized pivoted QR strategy on each matricized target tensor $\mathbf{T}_{k}$ rather than sampling based on the KRP $\mathbf{W}_{k}$.

\refstepcounter{algorithm}
\begin{figure*}[t]
\begin{mdframed}
\small

\textbf{Algorithm \thealgorithm:}

\begin{algorithmic}[1]

\State \textbf{Input:}
Tensor $\mathcal{T}\in\mathbb{R}^{n_1,\dots,n_N}$, rank $R$, maximum iterations
$t_{\max}$, sample sizes $\mathbf{s}$, tolerance $\epsilon$

\State \textbf{Output:}
Factor matrices $\{\mathbf{F}_1,\dots,\mathbf{F}_N\}$

\For{$k=1:N$}

    \State $s_k=\mathbf{s}(k)$
    \Comment{Extract number of samples for mode $k$}

    \State $(\mathbf{Q}_k,\mathbf{R}_k,\mathbf{\Pi}_k)
    =
    \mathrm{SE\text{-}QRCS}(\mathbf{T}_k)$
    \Comment{Apply randomized sRRQR}

    \State $\alpha_k=
    \left|\left\{i:\left|R_k(i,i)\right|\geq\epsilon\right\}\right|$

    \If{$s_k\leq\alpha_k$}

        \State $\mathbf{\Pi}_{s_k}
        =
        \mathbf{\Pi}_{:,1:s_k}$

    \Else

        \State $\mathbf{\Pi}_{s_k}
        =
        [
        \mathbf{\Pi}_{:,1:\alpha_k},
        \mathbf{\Pi}_{:,\sigma(\alpha_k+1:\mathrm{end})[1:s_k-\alpha_k]}
        ]$

        \Comment{Uniformly sample remaining columns}

    \EndIf

\EndFor

\For{$t=1:t_{\max}$}

    \For{$k=1:N$}

        \State $\mathbf{S}_k\mathbf{T}_k^T
        =
        \mathbf{\Pi}_{s_k}^{T}\mathbf{T}_k^T$

        \Comment{Compute sketched tensor}

        \State
        $\mathbf{F}_k=
        \arg\min_{\mathbf{F}}
        \|
        \mathbf{S}_k\mathbf{W}_k\mathbf{F}^{T}
        -
        \mathbf{S}_k\mathbf{T}_k^T
        \|_F^2$

        \Comment{Solve least squares problem}

    \EndFor

\EndFor

\end{algorithmic}

\end{mdframed}
\caption{ CPD-sRRQR Algorithm using the randomized sRRQR factorization, SE-QRCS.}
\label{Algorithm1}

\end{figure*}

In this section, we first provide a motivation for introducing the pivoted QR and, subsequently, we provide an efficient pivoted QR-based sampling algorithm for the CPD-ALS presented in Algorithm~\ref{Algorithm1}.

\begin{figure}[t!]
    \begin{subfigure}{0.5\textwidth}
        \includegraphics[width=0.8\columnwidth]{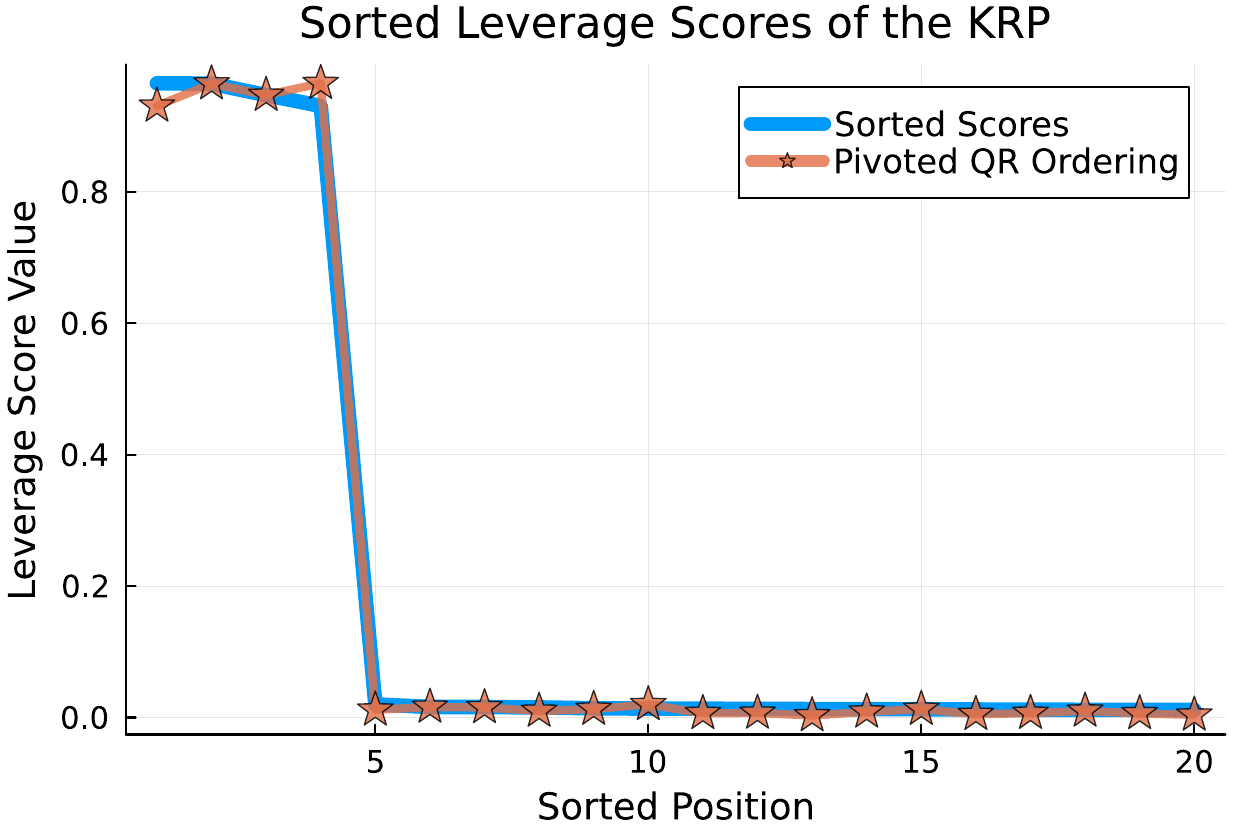}
        \caption{}
        \label{fig:AbsEnErr1}
    \end{subfigure}\hfill
    \begin{subfigure}{0.49\textwidth}
        \includegraphics[width=.8\linewidth]{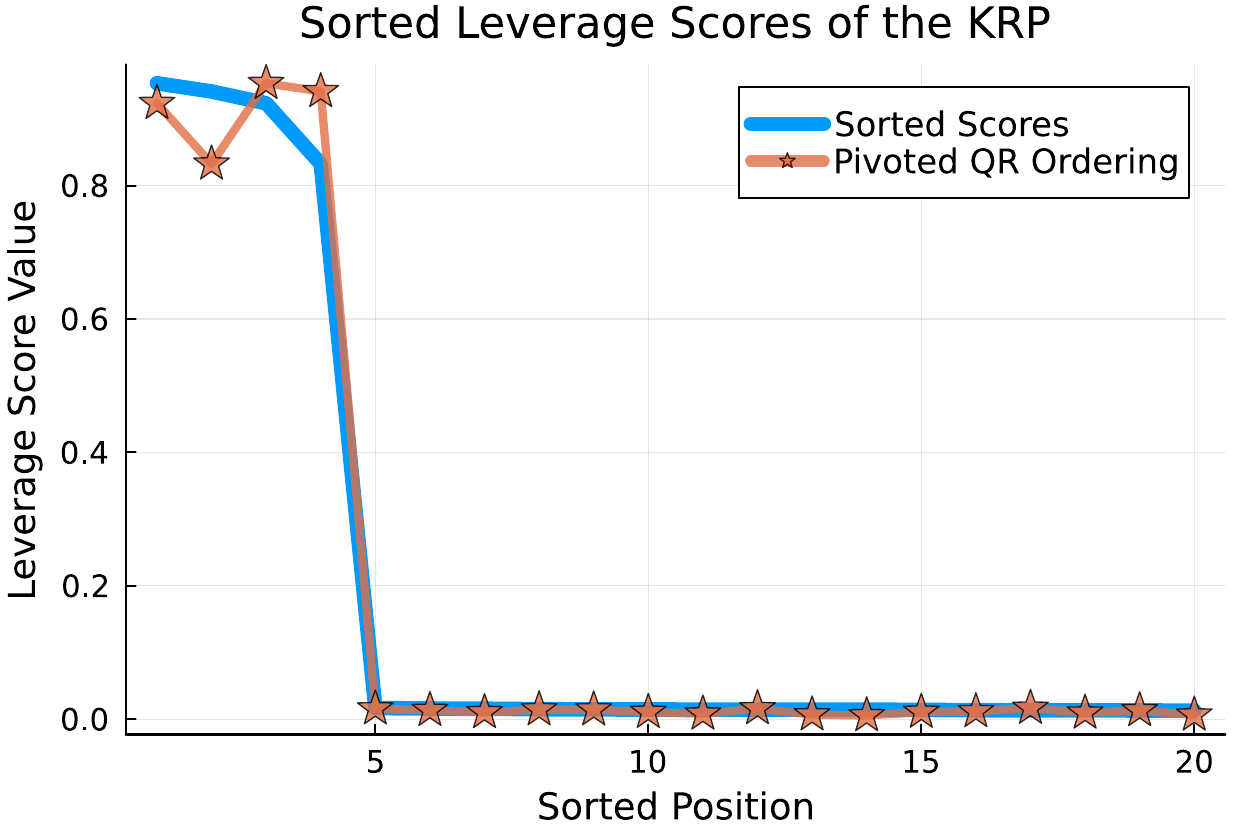}
        \caption{}
        \label{fig:DisEn1}
    \end{subfigure}\hfill
    \begin{subfigure}{0.5\textwidth}
        \includegraphics[width=.8\columnwidth]{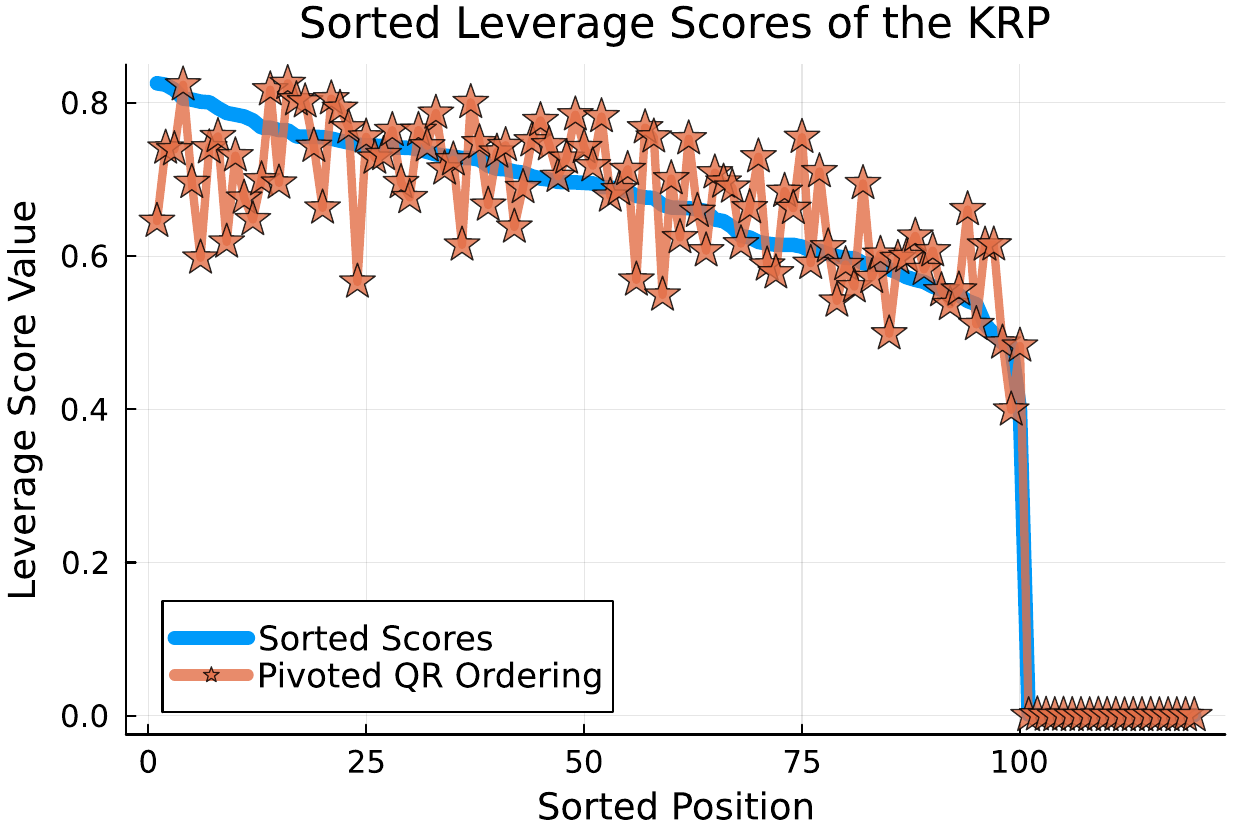}
        \caption{}
        \label{fig:AbsEnErr2}
    \end{subfigure}\hfill
    \begin{subfigure}{0.49\textwidth}
        \includegraphics[width=.8\linewidth]{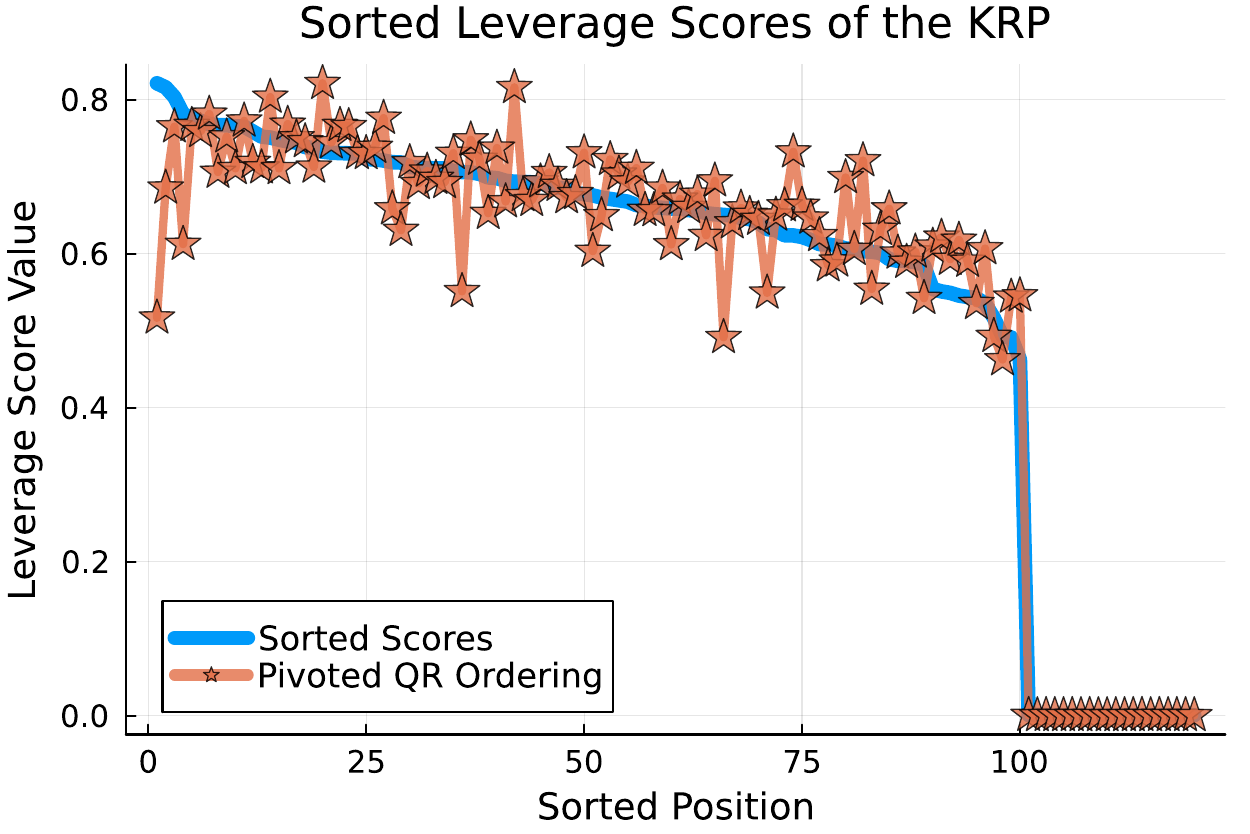}
        \caption{}
        \label{fig:DisEn2}
    \end{subfigure}\hfill
\caption{Sorted leverage scores of the KRP from an order-3 tensor synthetic tensor with dimensions (531, 308, 640). (a) and (b) represent data from a random tensor of CP rank 10 where 2 large leverage score are embedded into each factor matrix.
(c) and (d) represent data from a random tensor of CP rank 70 where 10 large condition numbers are embedded into each factor matrix. (a) and (c) contain the leverage scores of $W^*_1$ and (b) and (d) contain the leverage scores of $W^*_2$.}
\label{fig:lev_orders}
\end{figure}
\subsection{Samples from the Pivoted QR Scheme}
\label{QR_sampling_section}
Given the existence of the exact rank-$R$ CPD for a tensor $\mathcal{T}\in \mathbb{R}^{n_1,\times n_2\times\dots\times n_N}$ where, for convenience, $R < n_1, n_2, ... n_N$, we can express the mode-$k$ unfolding, $\mathbf{T}_k^T$, in the following way
\begin{align}
    \mathcal{T} &\equiv [\![ {\bf F}_1^*, {\bf F}_2^*, \dots {\bf F}_N^*  ] \!]\\ \nonumber
    {\bf T}^{T}_{k} &= {\bf W}^{*}_k [{\bf F}^*_{k}]^{T} \\ \nonumber
\end{align}
where $[ {\bf F}_1^*, {\bf F}_2^*, \dots {\bf F}_N^*  ]$ is a set of unknown factor matrices.
Our goal is to determine a single collection of good row samples using the leverage scores of $\mathbf{W}^*_k$. However, the exact Khatri--Rao product matrix $\mathbf W_k^*$ is unknown. We therefore seek to relate its leverage scores to those of the observable unfolding $\mathbf T_k^T$. To see how they are linked, let
$\mathbf{W}_{k}^{*}=\mathbf{U}_{W}\mathbf{\Sigma}_{W}\mathbf{V}_{W}^{T}$ be the SVD factorization of $\mathbf{W}_{k}^{*}$, where $\mathbf{U}_{W}\in \mathbb{R}^{m\times R}$. Then 
\begin{align*}
    \mathbf{T}_{k}^{T}
    &=\mathbf{W}_{k}^{*}\mathbf{F}_{k}^{*T}
    =\mathbf{U}_{W}\mathbf{\Sigma}_{W}\mathbf{V}_{W}^{T}\mathbf{F}_{k}^{*T}\\
    &=\mathbf{U}_{W}\bar{\mathbf{U}}\bar{\mathbf{\Sigma}}\bar{\mathbf{V}}^{T}
    =\mathbf{U}_{T}\bar{\mathbf{\Sigma}}\bar{\mathbf{V}}^{T}
\end{align*}
where in the second equality we denote by $\bar{\mathbf{U}}\bar{\mathbf{\Sigma}}\bar{\mathbf{V}}^{T}$ the SVD of
$\mathbf{\Sigma}_{W}\mathbf{V}_{W}^{T}\mathbf{F}_{k}^{*T}\in \mathbb{R}^{R \times n_k}$.
Assume that $R\le n_k$ and that the factor matrix $\mathbf F_k^*$ has full column rank $R$. Then we have
$\bar{\mathbf{U}}\in \mathbb{R}^{R \times R}$, $\bar{\mathbf{\Sigma}}\in \mathbb{R}^{R \times n_k}$, and $\bar{\mathbf{V}}\in \mathbb{R}^{n_k \times n_k}$, which implies $\mathbf{U}_{T}\in \mathbb{R}^{m\times R}$. This implies that the final factorization $\mathbf{U}_{T}\bar{\mathbf{\Sigma}}\bar{\mathbf{V}}^{T}$ is an SVD for $\mathbf{T}^{T}_{k}$.
This shows that in this case, the leverage scores of $\mathbf{W}_{k}^{*}$,
\[
l_{i}(\mathbf{W}_{k}^{*})=\|(\mathbf{U}_{W})_{i,:}\|_{2}^{2},
\]
satisfy
\[
l_{i}(\mathbf{W}_{k}^{*})
=\|(\mathbf{U}_{W})_{i,:}\|_{2}^{2}
=\|(\mathbf{U}_{W}\bar{\mathbf{U}})_{i,:}\|_{2}^{2},
\]
and therefore are also the leverage scores of $\mathbf{T}_{k}^{T}$, since $\mathbf{U}_{T}=\mathbf{U}_{W}\bar{\mathbf{U}}$.
For the second case $R>n_k$, we have $\mathbf{U}_{T}=\mathbf{U}_{W}\bar{\mathbf{U}}\in \mathbb{R}^{m \times n_k}$. 
This implies that the leverage scores of $\mathbf{T}_{k}^{T}$ correspond to the leverage scores of a rank-$n_k$ orthogonal projection of $\mathbf{W}_{k}^{*}$ onto the subspace spanned by $\mathbf{U}_{W}\bar{\mathbf{U}}$. In this case, we do not recover the exact leverage scores of $\mathbf{W}_{k}^{*}$, however, they are only linked through this projected subspace. But in practice, we can see that the sampling from $\mathbf{T}_{k}^{T}$ still works as in Figure \ref{fig:lev_orders_r_big} which shows results for $R>n_k$.
 
\begin{figure}[t!]
\begin{subfigure}{0.5\textwidth}
    \includegraphics[width=0.8\columnwidth]{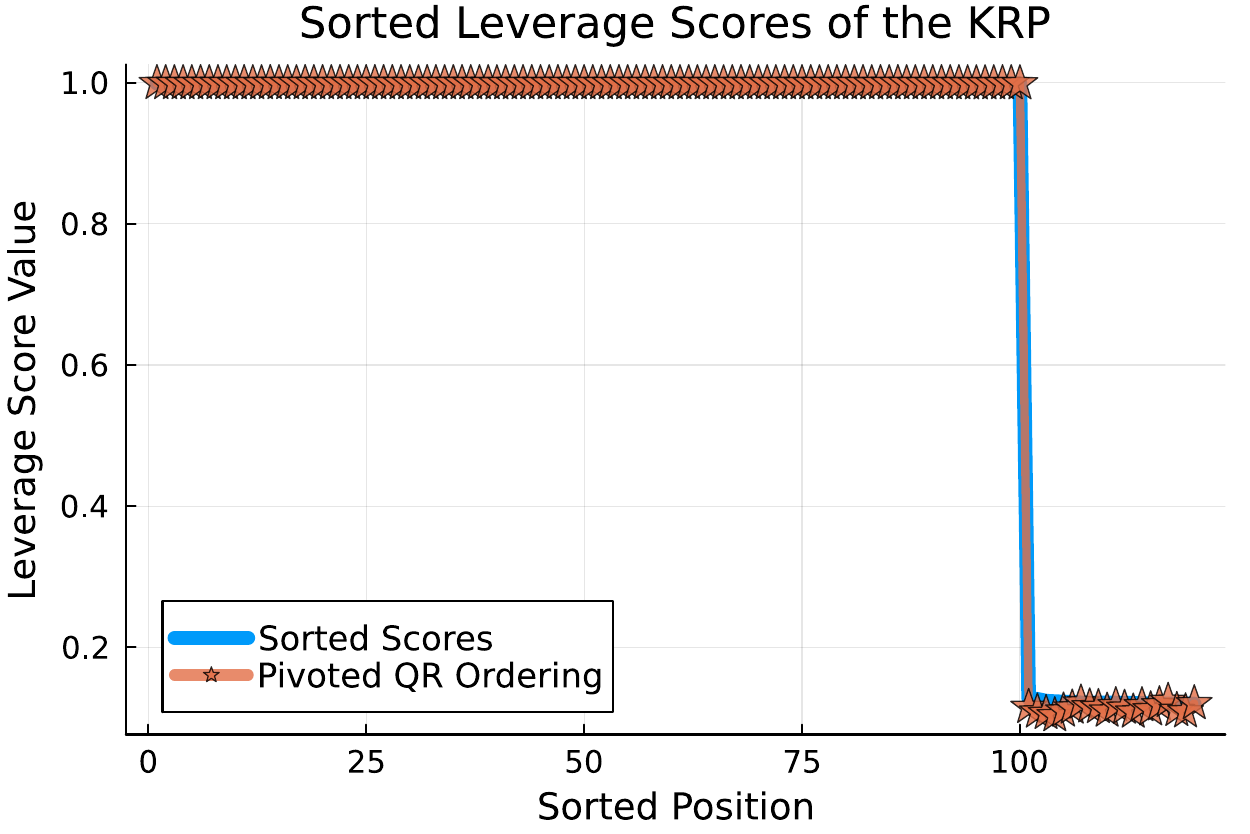}
    \caption{}
\end{subfigure}\hfill
\begin{subfigure}{0.49\textwidth}
    \includegraphics[width=.8\linewidth]{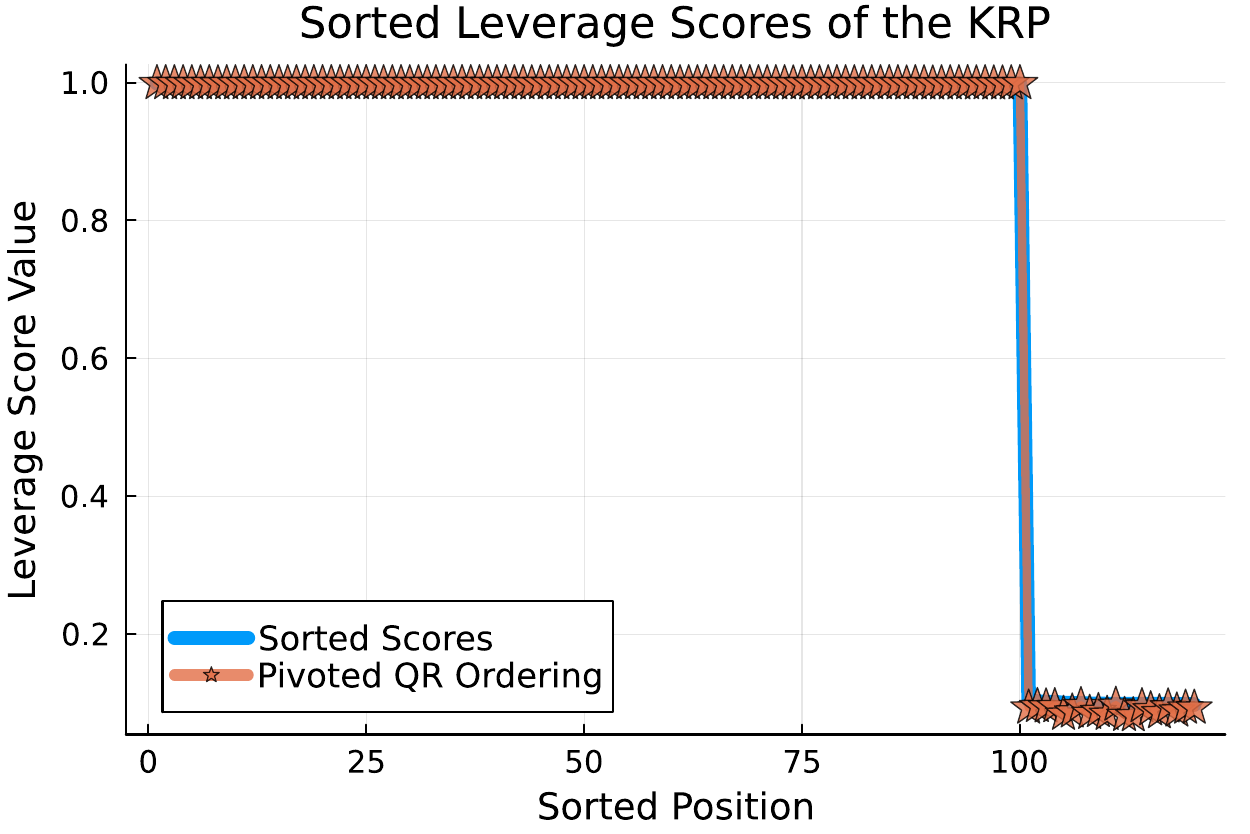}
    \caption{}
\end{subfigure}\hfill
\caption{Sorted leverage scores of the KRP from an order-3 tensor synthetic tensor with dimensions (531, 308, 640), a CP rank 1000, and where 10 large leverage condition numbers are embedded into each factor matrix. (a) considers the leverage scores of $W^*_1$ and (b) considers the leverage scores of $W^*_2$}
\label{fig:lev_orders_r_big}
\end{figure}

Although it is possible to find the leverage scores of the KRP using the SVD (or QR) of ${\bf T}^T_k$, the computational cost scale exponentially with the order of $\mathcal{T}$.
Therefore, this method would only be reasonable for relatively small tensors for which the SVD of $\mathbf{T}^T_k$ can be computed. 
In order to develop a sampling strategy that works for large tensors, instead of computing the exact row leverage scores of $\mathbf{T}_{k}^{T}$, we use column pivoted QR factorization on $\mathbf{T}_{k}$ to sample columns from the pivoting matrix. 
A demonstration of the pivoted QR's ability to identify the most statistically important rows of a known, exact KRP $\mathbf{W}_{k}^{*}$ is represented in \cref{fig:lev_orders}. 
In this figure, we construct order-3 tensors from a CPD where the true factor matrices each contain a known number of large leverage scores. 
We compute the exact leverage scores of KRP for each matricization of $\mathcal{T}$ and order the leverage scores from largest to smallest. 
We compare this optimal ordering to the order determined using the column-pivoted QR of the target tensor. 
We see that the column-pivoted QR of $\mathbf{T}_k$ can effectively order the leverage scores of $\mathbf{W}^*_k$.

As described in \cite{CSSP_Xue}, leverage-score sampling from $\mathbf{T}_{k}^{T}$ and column sampling via strong rank-revealing QR factorization of $\mathbf{T}_{k}$ can both be connected to Determinantal Point Processes (DPPs), where DPP is similar to volume sampling in randomized numerical linear algebra. We briefly describe this connection, for more details the reader can consult \cite{CSSP_Xue}.
Formally, let $\mathbf{L}\in \mathbb{R}^{m \times m}$ be a symmetric positive semidefinite matrix. A point process $S \subseteq [m]$ is sampled according to a determinantal point process, $S \sim \mathrm{DPP}(\mathbf{L})$, if
\[
\Pr(S) = \frac{\det(\mathbf{L}_S)}{\det(\mathbf{I}+\mathbf{L})}.
\]

To see the connection between DPP sampling and volume sampling, consider $\mathbf{L} = \mathbf{T}_{k}^{T}\mathbf{T}_{k} \in \mathbb{R}^{m \times m}$. In this case, the probability of sampling a subset $S \subseteq [m]$ under the corresponding $L$-ensemble DPP satisfies
\[
\Pr(S) \propto \det((\mathbf{T}_{k}^{T}\mathbf{T}_{k})_S)
= \mathrm{Vol}^{2}(\{\mathbf{t}_{i}: i \in S\}),
\]
where $\mathrm{Vol}(\{\mathbf{t}_{i}: i \in S\})$ denotes the $|S|$-dimensional volume of the parallelepiped spanned by the columns $\{\mathbf{t}_i\}_{i \in S}$ of $\mathbf{T}_{k}$.

From one perspective, it is shown in \cite{derezinski2021determinantal} that the marginal inclusion probabilities of a DPP are given by
\[
\mathbb{P}(i \in S) = \bigl(\mathbf{L}(\mathbf{I}+\mathbf{L})^{-1}\bigr)_{ii}.
\]

From another perspective, strong rank-revealing QR (sRRQR) selects a subset of columns that approximately maximizes the same volume objective, making it a deterministic analogue of volume-based sampling. In this sense, leverage-score sampling, DPPs, and sRRQR are closely connected through the shared underlying objective of selecting high-volume subsets.

\subsection{Hybrid Sampling for CP-sRRQR}
Here we propose two methods for determining row-wise samples for each LS subproblem \eqref{ALS_update} of the CPD-ALS:
\begin{enumerate}
    \item {\bf Deterministic Sampling}: take all $s$ samples from the pivot matrix in the order provided by the column pivot QR algorithm.
    \item {\bf Hybrid Sampling}: take the first $\alpha$ samples from the order provided by the pivot matrix and then uniformly sampling the remaining $s - \alpha$ columns from the pivots not chosen in the first $\alpha$ positions. $\alpha$ is chosen to be the rank of the $\mathbf{T}_{k}$ which can be determined as either $n_k$ or by considering the magnitude of the diagonal elements of $\mathbf{R}_{k}$ in the QR procedure. 
\end{enumerate}

Method $1.$ is a straightforward approach and assumes that the QR can order all leverages scores of the KRP. 
Unfortunately, the column-pivoted QR is a greedy algorithm that terminates after ordering, at most, $n_k$ columns of $\mathbf{T}_{k}$ and, in general, $n_k \ll m$ so the number of trustworthy pivots from the QR is relatively small. 
Furthermore, in our tests, we have found that the QR ordering algorithm may fail to recover {\bf all} intermediate valued leverage scores for large tensors, tensors of large CP rank, and tensors with a large number of statistically important rows in their KRPs.
As a means to correct this problem we propose method $2.$ a hybrid deterministic/probabilistic method.
In this work, for method $2.$ we determine $\alpha$ by approximating the rank of $\bf{T}_k$ using a relative truncation metric of the diagonal elements of $\mathbf{R}_{k}$.
As a means to compare these sampling procedures, we consider the COIL dataset.

{\textbf{COIL data:}}
COIL is an image-recognition dataset containing 100 images each in 72 different positions\cite{Nene:1996:ColumbiaOI}.
This set has been established as a tensor decomposition benchmark\cite{battaglino2018practical,larsen2022practical,Zhou:2018:Arxiv} due to its irregular dimensions and large size.
In this set, each image is 128 x 128 pixels and utilize three color channels associated with standard RGB. 
In the following results we consider only 20 of the 100 images with all single image rotations resulting in a tensor of dimension 1440x128x128x3, which requires 0.56GB of memory.
Please note, we choose 20 images because we saw no major difference in any results regardless of the number of selected images.
In our first test, we decompose this tensor to a fixed CP of rank 20 with a different number of row-wise samples.
We compare the results of the QR-based methods to the randomized leverage score-based sampling CPD-ALS algorithm introduced by Larsen et al.\cite{larsen2022practical}.
In \cref{fig:first_fixed_20}, we use the deterministic sampling technique and see very poor results for the QR-based sampled ALS methods whereas in \cref{fig:second_fixed_20} we use the hybrid sampling technique and see that our QR-based sampling algorithms find decompositions which are comparable to ones found using Larsen's leverage score-based sampled ALS algorithm.
\begin{figure}[H]
    \centering
    \begin{subfigure}[b]{0.45\textwidth}
        \includegraphics[width=\textwidth]{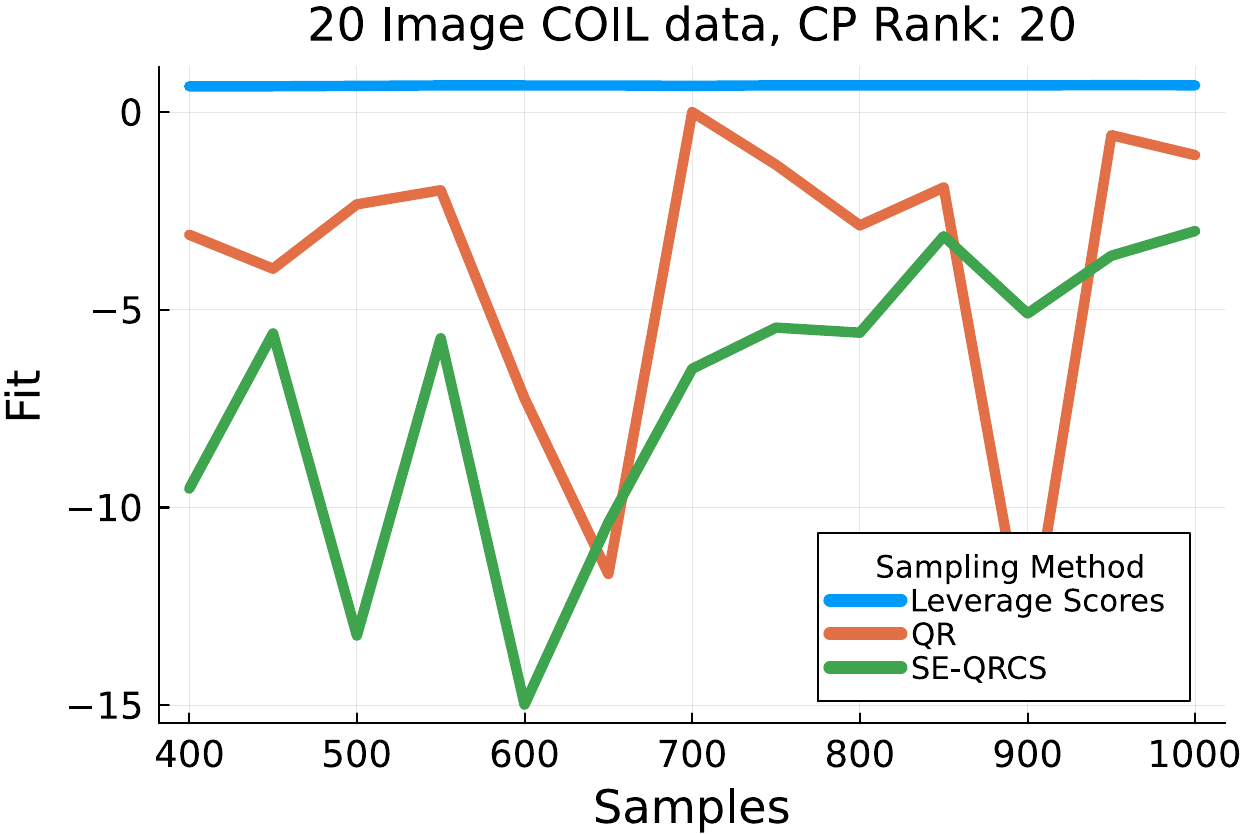}
        \caption{CPD fit of COIL data with a fixed rank using sampling method 1.}
        \label{fig:first_fixed_20}
    \end{subfigure}
    \hfill
    \begin{subfigure}[b]{0.45\textwidth}
        \includegraphics[width=\textwidth]{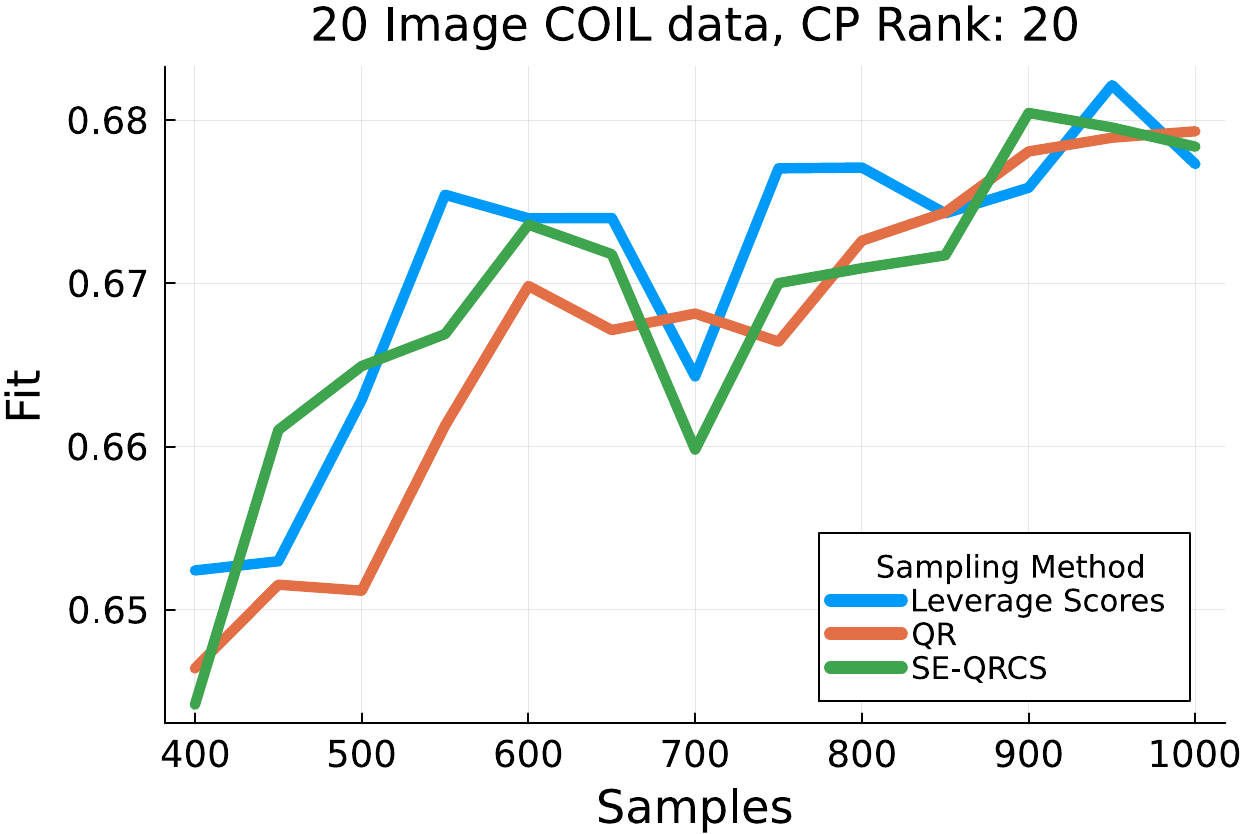}
        \caption{CPD fit of COIL data with a fixed rank using sampling method 2.}
        \label{fig:second_fixed_20}
    \end{subfigure}
\end{figure}
We repeat this experiment in \cref{fig:first_fix_samples} and \cref{fig:second_fix_samples} where we, now, fix the number of row sample to $s = 1000$ and change the CP rank from 20 to 300. 
These figures further corroborate the rationale to use the hybrid sampling technique.
\begin{figure}[H]
    \begin{subfigure}[b]{0.45\textwidth}
        \includegraphics[width=\textwidth]{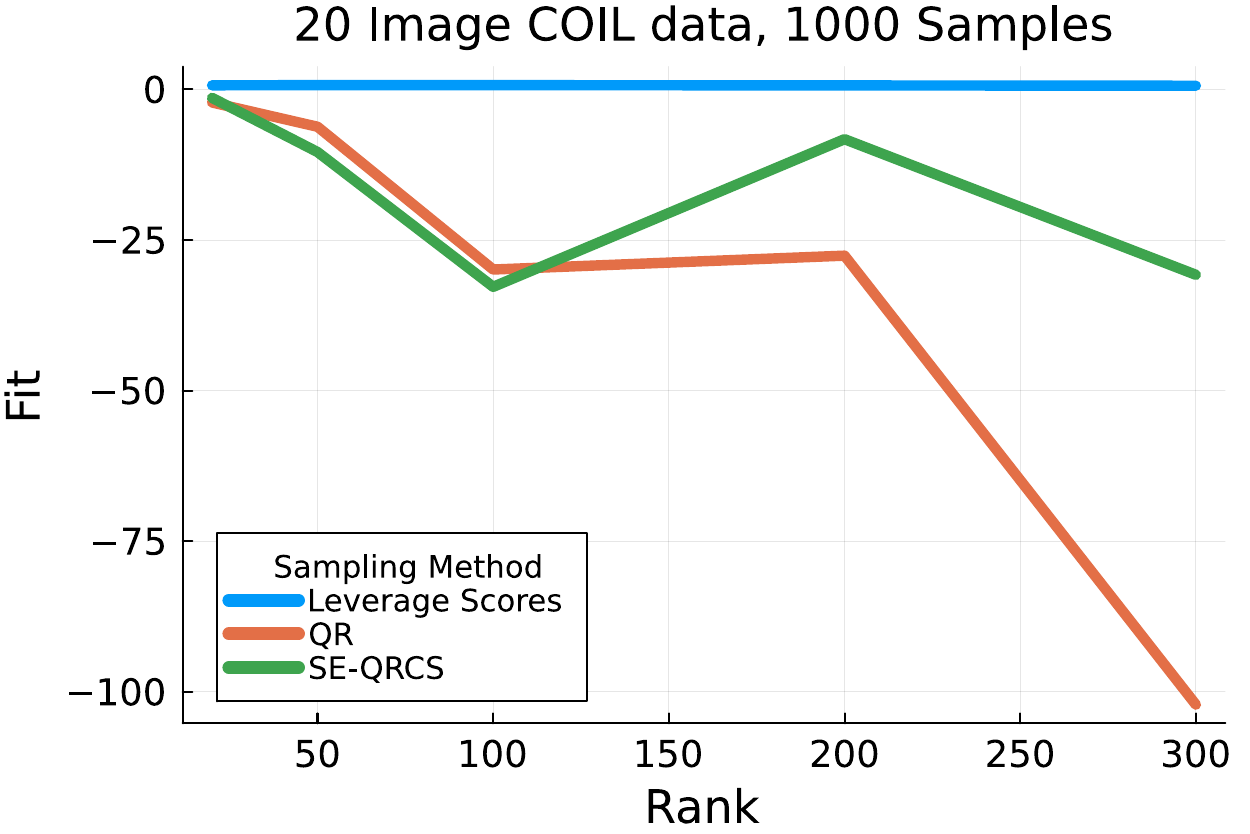}
        \caption{CPD fit of COIL data with a fixed number of samples using sampling method 1.}
        \label{fig:first_fix_samples}
    \end{subfigure}\hfill
    \begin{subfigure}[b]{0.45\textwidth}
        \includegraphics[width=\textwidth]{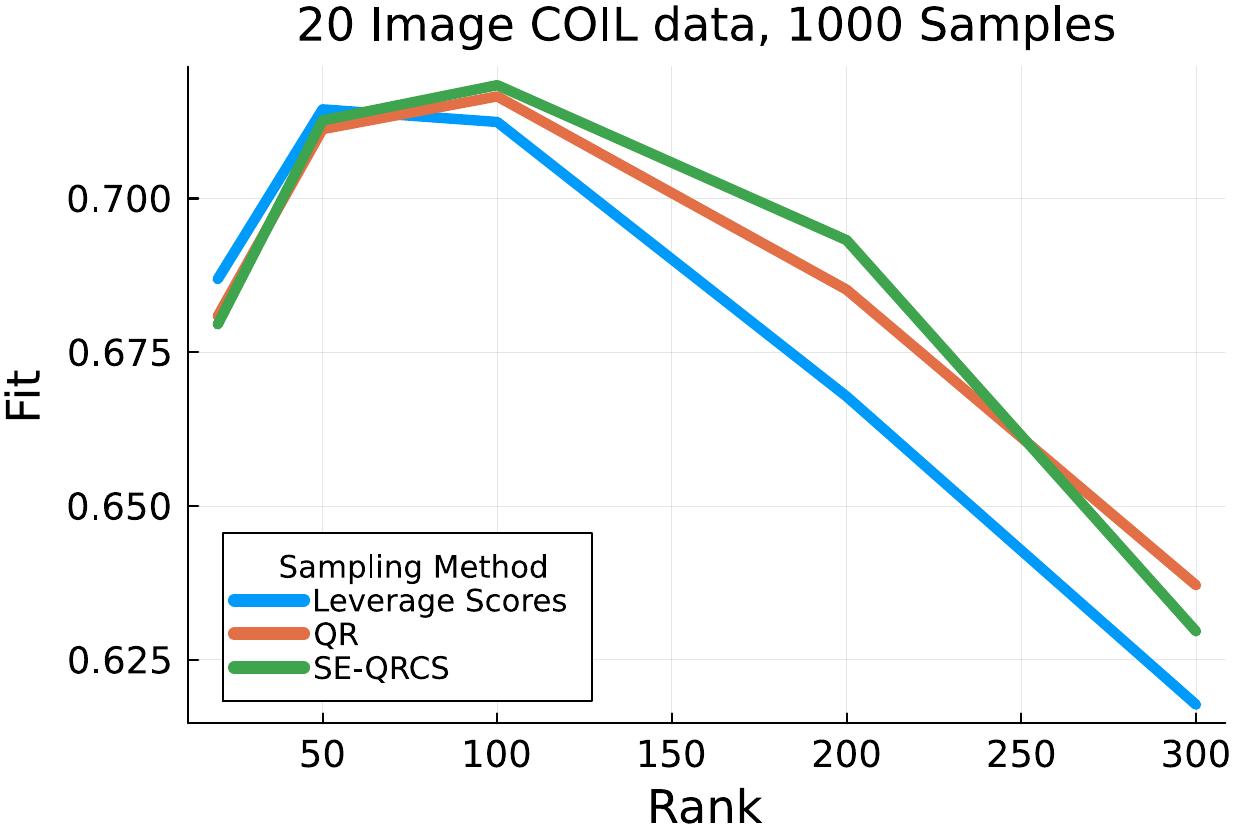}
        \caption{CPD fit of COIL data with a fixed number of samples using sampling method 2.}
        \label{fig:second_fix_samples}
    \end{subfigure}
\end{figure}

In Algorithm~\ref{Algorithm1}, we present the CPD-sRRQR algorithm that uses Hybrid sampling to sketch the least square problem given in CPD \eqref{ALS_update}. Due to the high computational cost of the deterministic sRRQR, we use instead the randomized sRRQR method, SE-QRCS.
\subsection{Computational Complexity}
The computational complexity of the QR-based sampled CPD-ALS algorithm consists of two main parts. 
The first part is the application of the SE-QRCS method to each matricization $\mathbf{T}_{k}$ and the construction of the sampled target tensors $\mathbf{S}_k \mathbf{T}_{k}^T$.
For each mode $k$, the complexity of SE-QRCS \cite{fakih2025efficient} is given by $O(mn_{k}\tau+n_{k}(p+l)\alpha)$ where the first term corresponds to the complexity of applying a sparse embedding to $\mathbf{T}_{k}$ and the second term corresponds to applying QR factorization to the sketched matrix $(\mathbf{T}_{k})_{sk}$ and the reduced column subset matrix $(\mathbf{T}_{k})_{sub}^{1}$.
After determining the pivot matrix, the complexity of sampling the $k$-th unfolding of the target tensor is $\mathcal{O}(n_k s)$.
As $\mathcal{T}$ is not resampled in the ALS algorithm, these sampled matricizations can be stored instead of the full $\mathcal{T}$ tensor. 
This replacement may reduce the computational storage complexity from $\mathcal{O}(\prod^N_1 n_k)$ to $\mathcal{O}(\sum_1^Nn_ks)$.

The second part of our QR-based sampled CPD-ALS algorithm is the cost associated with the sampled ALS optimization via the construction of the sampled KRP and the LS solve. 
The computational complexity of constructing the KRP for each LS subproblem using the sampled KRP function developed by Battaglino et al.~\cite{battaglino2018practical} is $\mathcal{O}((N-1)sR)$.
To solve the LS problem, we propose two common algorithms. 
The first algorithm solves the equation
\[\mathbf{S}_{k}\mathbf{W}_{k}\mathbf{F}_{k}^{T}=\mathbf{S}_{k}\mathbf{T}_{k}^{T}\]
using a column-pivoted QR matrix-solve algorithm. This algorithm consists of computing the column-pivoted QR factorization of $\mathbf{S}_{k}\mathbf{W}_{k}$ and, subsequently, sample the matricized tensor $\mathbf{S}_{k}\mathbf{T}_{k}^{T}$. The computational complexity of these operations is $O(n_{k}sR+sR^{2})$.
The second algorithm solves the sketched normal equations, 
\begin{align}
\mathbf{W}_{k}^{T}\mathbf{S}_{k}^{T}\mathbf{S}_{k}\mathbf{W}_{k}\mathbf{F}_{k}^{T}=\mathbf{W}_{k}^{T}\mathbf{S}_{k}^{T}\mathbf{S}_{k}\mathbf{T}_{k}^{T}.
\end{align}
The complexity of this algorithm is $\mathcal{O}(sR^2 + n_ksR + R^3)$, which can be split into the computations of $\mathbf{W}_{k}^{T}\mathbf{S}_{k}^{T}\mathbf{S}_{k}\mathbf{W}_{k}$, $\mathbf{W}_{k}^{T}\mathbf{S}_{k}^{T}\mathbf{S}_{k}\mathbf{T}_{k}^{T}$, and a Cholesky-based matrix solve, respectively.
Although this second equation does require more work to both compute the Grammian matrix and the sketched matricized tensor times KRP, we find that the most time intensive step in the LS solve is the matrix-solve.
Furthermore, we find that Cholesky-based linear solve algorithms are significantly more efficient than QR-based algorithms.
We illustrate the performance benefits of using the normal equation based algorithm in \cref{fig:normal} and in
\cref{tab:LS_solvers} we demonstrate the numerical stability of the normal-equation based ALS solver.

\begin{table}[H]
\centering
\begin{tabular}{|c|c|c|c|}
\hline
  Tests & $N_\mathrm{samples}$ & LS Solve Fit & Normal Equations Fit \\ \hline
1  & 200  & 0.378 & 0.494   \\ \hline
2  & 1000  & 0.722 & 0.722  \\ \hline
3  & 2000   & 0.673 & 0.673 \\ \hline
\end{tabular}
\caption{Comparing the quality of fit using two different solver algorithms for a random $(90 \times 90 \times 90)$ tensor constructed from a CP of rank 100 and modified elements as described in \cref{sec:comp}.}
\label{tab:LS_solvers}
\end{table}

\begin{figure}
    \centering
    \includegraphics[width=0.5\linewidth]{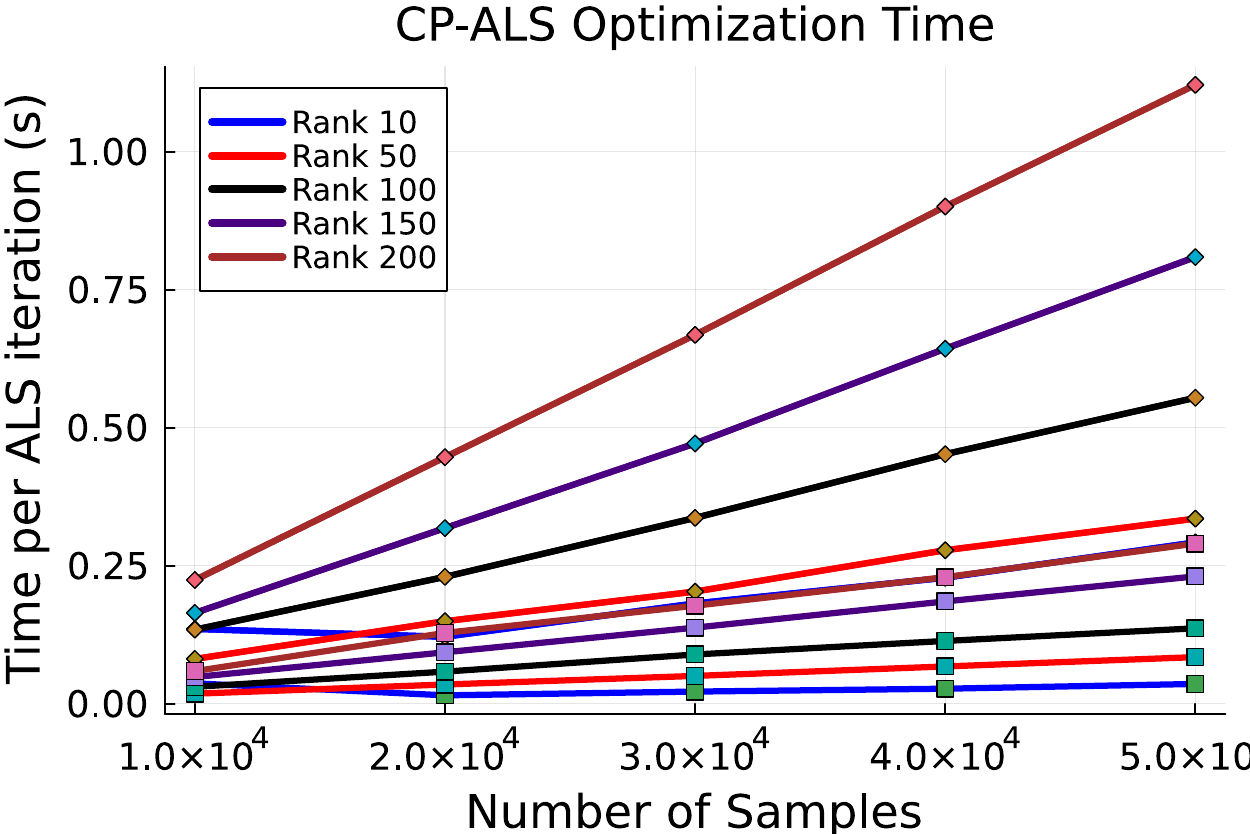}
    \caption{Per iteration timing results comparing the standard sampled least-squares based optimization strategy (diamond markers) to the sampled normal-equation based strategy (square markers) for a $500 \times 500 \times 500$ tensor.}
    \label{fig:normal}
\end{figure}

%% file: Theoretical_Results.tex
\section{Theoretical Guarantees}
\subsection{Leverage Score Sampling}
In this section, we present the number of samples needed for each least square subproblem \eqref{ALS_update}, such that the sketched least square subproblem achieves $(1+\epsilon)$-relative accuracy in the residual norm. First we note that in this section we avoid the index referring to the mode, as the analysis holds for every ALS problem. In our setting, we are given the matrices $\mathbf{W}\in \mathbb{R}^{m\times R}$ and $\mathbf{T}\in \mathbb{R}^{m\times n}$ and we define the overdetermined least squares problem, 
    \[\mathcal{R}^{2}=\min_{\mathbf{F}\in \mathbb{R}^{n\times R}}\|\mathbf{WF}^{T}-\mathbf{T}\|_{F}^{2}.\]
Given the non-convex nature of the ALS objective function,  we assume in our setting that a good initialization is chosen such that it guarantees the convergence of the iterative process to the global solution rather than becoming trapped in stationary points \cite{Uschmajew:2012:SJMAA, kolda2009tensor}. Let $\mathbf{W}=\mathbf{U}_{W}\mathbf{\Sigma}_{W}\mathbf{V}_{W}^{T}$ 
and $\mathbf{T}=\bf{W^{*}F^{*}}$  with $\mathbf{W^{*}}=\mathbf{U}_{W^{*}}\mathbf{\Sigma}_{W^{*}} \mathbf{V}_{W^{*}}^{T}\in \mathbb{R}^{m\times R}$ be the SVD factorization of $\mathbf{W}$ and $\mathbf{W^{*}}$ where $\mathbf{U}_{W}$ and $\mathbf{U}_{W^{*}}$ are the orthonormal basis for the columns spaces of $\mathbf{W}$ and $\mathbf{W^{*}}$ respectively. Denote by $\mathbf{U}_{W}^{\perp}$ an orthonormal basis for the subspace orthogonal to the column space of $\mathbf{W}$ and $\mathbf{T}^{\perp}=\mathbf{U}_{W^{\perp}}(\mathbf{U}_{W^{\perp}})^{T}\mathbf{T}$ the projection of the columns of $\mathbf{T}$ to this orthogonal subspace.
Then the least square problem can be represented in terms of $\mathbf{T^{\perp}}$, 
\[\mathcal{R}^{2}=\min_{\mathbf{F}\in \mathbb{R}^{n\times R}}\|\mathbf{WF}^{T}-\mathbf{T}\|_{F}^{2}=\|\mathbf{U}_{W}^{\perp}(\mathbf{U}_{W}^{\perp})^{T}\mathbf{T}\|_{F}^{2}=\|\mathbf{T}^{\perp}\|_{F}^{2}\]

Let $\mathbf{S}\in \mathbb{R}^{s\times m}$ be a row selection matrix defined by 
\[\mathbf{S}_{ij}=\begin{cases}
    \frac{1}{\sqrt{sp_{i}}} \quad \text{if }\psi_{i}=j \\
    0 \quad \text{otherwise}
\end{cases}\]
where $\psi_{i}$ are i.i.d samples from a distribution $\mathbf{p}$ on $[m]$ according to row leverage scores of $\bf{W^{*}}$. To ensure later that the matrix multiplication estimator in Lemma \ref{matrix approximation for W} remains unbiased, we introduce a small smoothing parameter $\varepsilon \in (0, 1)$ if $\mathbf{p}_{i}=0$ for some $i$. We define the regularized sampling probabilities $\tilde{p}_i$ as a mixture of the original probabilities and a uniform distribution:
\begin{equation}
\label{perturbed probabilities}
\tilde{p}_i = (1 - \varepsilon)p_i + \frac{\varepsilon}{m}
\end{equation}
For simplicity of presentation, we assume henceforth that the row leverage scores of $\mathbf{W}^{*}$ are strictly positive (i.e., $p_{i} > 0$ for all $i$). Otherwise, they can be implicitly replaced by the regularized probabilities defined in \eqref{perturbed probabilities} 
 Define the sketched least squares problem, 
\begin{equation}
\label{sketched_least_square}
\min_{\mathbf{F}\in \mathbb{R}^{n\times R}}\|\mathbf{S WF}^{T} -\mathbf{S W^{*}F^{*}}\|_{F}^{2}.
\end{equation}
Following a similar proof as in \cite{larsen2022practical}, the sketched least squares \eqref{sketched_least_square} approximates the original least squares up to tolerance $\epsilon>0$ if two structural conditions are satisfied on the sketched matrix $\mathbf{S}$. The two structural conditions are summarized in the next Theorem \cite{drineas2011faster}.

\begin{theorem}[Theorem 10 in \cite{larsen2022practical}]
\label{Sketched_lS_Theorem}
    Let $\mathbf{W}\in \mathbb{R}^{m\times R}$ and $\mathbf{T}\in \mathbb{R}^{m\times n}$ and $\mathbf{S}\in \mathbb{R}^{s\times m}$. Assume that $\mathbf{S}$ satisfies the two structural conditions for some $\epsilon>0$
    \begin{equation}
    \label{structural_cond_1}
        \sigma_{\min}^{2}(\mathbf{S U}_{W})\geq 1/\sqrt{2}
    \end{equation}
    \begin{equation}
        \label{structural_cond_2}
        \|\mathbf{U}_{W}^{T}\mathbf{S}^{T}\mathbf{S T}^{\perp}\|_{F}^{2}\leq \epsilon\mathcal{R}^{2}/2
    \end{equation}
    Then the solution to the sketched problem \eqref{sketched_least_square}, $\mathbf{F_{sk}}$, and the optimal solution $\mathbf{F_{opt}}$ satisfies the following two bounds:
\begin{equation*}
\begin{aligned}
    \|\mathbf{W}\mathbf{F}_{sk}^{T}-\mathbf{T}\|_{F}^{2}
    \le (1+\epsilon)\|\mathbf{WF}_{opt}^{T}-\mathbf{T}\|_{F}^{2} \\[6pt]
    \|\mathbf{F}_{opt}-\mathbf{F}_{sk}\|_{F}^{2}
    \le \frac{\epsilon\,\|\mathbf{WF}_{opt}^{T}-\mathbf{T}\|_{F}^{2}}{\sigma_{\min}^{2}(\mathbf{W})}
\end{aligned}
\end{equation*}
\end{theorem}

\begin{lemma}
\label{embedding for W}   
 Let $\epsilon,\delta>0$, $\mathbf{W}=\mathbf{U}_{W}\mathbf{\Sigma}_{W} \mathbf{V}_{W}^{T}\in \mathbb{R}^{m\times R}$ and $\mathbf{W^{*}}\in \mathbb{R}^{m\times R}$ with row leverage scores $l_{i}(\mathbf{W^{*}})$. If $\mathbf{S}\in \mathbb{R}^{s\times m}$ is a scaled row sampling matrix according to a probability distribution $\mathbf{p}\in [0,1]^{m}$, $\mathbf{p}_{i}\geq \beta l_{i}(\mathbf{W^{*}})/R$ for some $\beta\in (0,1]$. Then for  $s>144\gamma R\ln(2R/\delta)/(\beta \epsilon^{2})$,where $\gamma=\max_{i}\left(\frac{\|(\mathbf{U}_{W})_{i,:}\|_{2}^{2}}{\|(\mathbf{U}_{W^{*}})_{i,:}\|_{2}^{2}}\right) $, 
 then with probability at least $1-\delta$, $1-\epsilon\leq \sigma_{i}^{2}(\mathbf{S} {\mathbf{U}_{W}})\leq 1+\epsilon$
\end{lemma}
\begin{proof}
    The proof follows using Bernstein inequality. 

    Define $\mathbf{I}-\left(\mathbf{SU}_{W}\right)^{T}\left(\mathbf{SU}_{W}\right)=\sum_{t=1}^{s}\mathbf{Z}_{t}=\sum_{t=1}^{s}(\frac{1}{s}\mathbf{I}-\mathbf{X}_{t})$ where $\mathbf{X}_{t}=\frac{1}{s\mathbf{p}_{i_{t}}}(\mathbf{U}_{W})_{i_{t},:}^{T}(\mathbf{U}_{W})_{i_{t},:}$ where $i_{t}$ is the index sampled at step $t$. Then we have the following two points: 
    \begin{itemize}
        \item $\mathbb{E}\left[\mathbf{Z}_{t}\right] = \frac{1}{s}\mathbf{I}-\frac{1}{s}\mathbf{U}_{W}^{T}\mathbf{U}_{W}=0$
        \item $\left\|\mathbf{Z}_{t}\right\|_{2}\leq \frac{1}{s}+\frac{1}{s p_{i_{t}}}\|(\mathbf{U}_{W})_{i_{t},:}\|^{2}_{2} \leq \frac{1}{s}(1+\frac{R}{\beta}\gamma)$, where $\gamma =\max_{i \\ /p_{i}\neq 0}\left(\frac{\|(\mathbf{U}_{W})_{i,:}\|_{2}^{2}}{\|(\mathbf{U}_{W^{*}})_{i,:}\|_{2}^{2}}\right)$. Let $L = \frac{1}{s}(1+\frac{R}{\beta}\gamma)$.
        \item  Using independence of the variable $\mathbf{X}_{t}$ we have that $\sigma^{2} = \|\mathbb{E}(\sum_{t=1}^{s}\mathbf{Z}_{t})^{2}\|_{2} =\left\|\sum_{t=1}^{s}\mathbb{E}[\mathbf{Z}_{t}^{2}]\right\|_{2}\leq \sum_{t=1}^{s}\|\mathbb{E}[\mathbf{Z}_{t}^{2}]\|_{2}$. We start first by calculating $\mathbb{E}[\mathbf{X}_{t}^{2}]$. 
        \begin{align*}
            \mathbb{E}[\mathbf{X}_{t}^{2}]&=\frac{1}{s^{2}}\sum_{i=1}^{m}\frac{1}{\mathbf{p}_{i}}((\mathbf{U}_{W})_{i,:}^{T}(\mathbf{U}_{W})_{i,:})((\mathbf{U}_{W})_{i,:}^{T}(\mathbf{U}_{W})_{i,:}) \leq \frac{\gamma R}{\beta s^{2}}\mathbf{I}
        \end{align*}
        As $\mathbb{E}[\mathbf{Z}_{t}^{2}]=\mathbb{E}[\mathbf{X}_{t}^{2}]-\left(\mathbb{E}[\mathbf{X}_{t}]\right)^{2}\leq \frac{1}{s^{2}}\left(\frac{\gamma R}{\beta }-1\right)\mathbf{I}$, the final variance is $\sigma^{2}\leq \frac{1}{s}\left|\frac{\gamma R}{\beta}-1\right|$
    \end{itemize}
    Using the above parameters and the Bernstein inequality we have 
    \[P\left\{\left\|(\mathbf{SU}_{W})^{T}\mathbf{SU}_{W}-\mathbf{I}\right\|_{2}\geq t \right\}\leq (2R)\exp\left(\frac{-t^{2}/2}{\sigma^{2}+Lt/3}\right)\] For $t=\epsilon$ and the given value of $s$, we have then with probability at least $1-\delta$, $1-\epsilon\leq \sigma_{i}^{2}(\mathbf{S} \mathbf{U}_{W})\leq 1+\epsilon$ for $i=1$. The inequality then follows for all $i$ using Theorem 2.2 in \cite{grigori2025randomized}.
\end{proof}
\begin{lemma}
    \label{SC1_W}
    Let $\epsilon,\delta>0$, $\mathbf{W}=\mathbf{U}_{W}\mathbf{\Sigma}_{W} \mathbf{V}_{W}^{T}\in \mathbb{R}^{m\times R}$ and $\mathbf{W^{*}}\in \mathbb{R}^{m\times R}$ with row leverage scores $l_{i}(\mathbf{W^{*}})$. Let $\mathbf{S}\in \mathbb{R}^{s\times m}$ is a scaled row sampling matrix according to a probability distribution $\mathbf{p}\in [0,1]^{m}$, $\mathbf{p}_{i}\geq \beta l_{i}(\mathbf{W^{*}})/R$ for some $\beta\in (0,1]$. Then for  $s>144\gamma R\ln(2R/\delta)/(\beta \epsilon_{0}^{2})$, where $\gamma = \max_{i}\left(\frac{\|(\mathbf{U}_{W})_{i,:}\|_{2}^{2}}{\|(\mathbf{U}_{W^{*}})_{i,:}\|_{2}^{2}}\right) $ and $\epsilon_{0}=1-\frac{1}{\sqrt{2}}$,
, $\sigma_{min}^{2}(\mathbf{S}\mathbf{U}_{W})\geq \frac{1}{\sqrt{2}}$ with probability at least $1-\delta$
\end{lemma}
\begin{proof}
    The proof follows by applying Lemma \ref{embedding for W} for the given $\epsilon_{0}$
\end{proof}
For the second structural condition, we follow a similar approach as in \cite{larsen2022practical}. We start first by stating the expected value of the matrix-matrix multiplication  approximation by sampling according to probabilities $\mathbf{p}$ as stated in \cite{drineas2006fast}.
\begin{lemma} [Lemma  4 in \cite{drineas2006fast}]
    \label{bound om expected value}
    Let $\mathbf{A}\in \mathbb{R}^{R \times m}$ and $\mathbf{B}\in \mathbb{R}^{m \times n} $ and $\mathbf{p}_{i}$ be a probability distribution. Construct a scaled sampling matrix $S\in \mathbb{R}^{s \times m}$ such that it samples $s$ rows i.i.d according to the probability distribution $\mathbf{p}$ to form the approximate matrix multiplication $\mathbf{AS}^{T}\mathbf{SB}$ to $\mathbf{AB}$. Then 
    \[\mathbb{E}\left[\left\|\mathbf{AS}^{T}\mathbf{SB}-\mathbf{AB}\right\|_{F}^{2}\right]=\sum_{i=1}^{m}\frac{\|\mathbf{A}_{:,i}\|_{2}^{2}\|\mathbf{B}_{i,:}\|_{2}^{2}}{s \mathbf{p}_{i}}-\frac{1}{s}\|\mathbf{AB}\|_{F}^{2}\]
\end{lemma}
\begin{lemma} 
\label{matrix approximation for W}
     Let $\delta>0$, $\mathbf{W}=\mathbf{U}_{W}\mathbf{\Sigma}_{W} \mathbf{V}_{W}^{T}\in \mathbb{R}^{m\times R}$ and $\mathbf{W^{*}}\in \mathbb{R}^{m\times R}$ with row leverage scores $l_{i}(\mathbf{W^{*}})$. If $\mathbf{S}\in \mathbb{R}^{s\times m}$ is a scaled row sampling matrix according to a probability distribution $\mathbf{p}\in [0,1]^{m}$, $\mathbf{p}_{i}\geq \beta l_{i}(\mathbf{W^{*}})/R$ for some $\beta\in (0,1]$. If $s\geq \frac{2R\gamma}{\beta \delta \epsilon}$, where $\gamma = \max_{i}\frac{\|\mathbf{(U}_{W})_{i,:}\|_{2}^{2}}{\|(\mathbf{U}_{W^{*}})_{i,:}\|_{2}^{2}}$ then with probability at least $1-\delta$, 
\[\left\|\mathbf{U}_{W}^{T}\mathbf{S}^{T}\mathbf{S}\mathbf{T}^{\perp}\right\|_{F}^{2}\leq\frac{\epsilon \|\mathbf{T}^{\perp}\|_{F}^{2}}{2}\]
\end{lemma}
\begin{proof}
   We have the following points.
    \begin{itemize}
        \item Using Lemma \ref{bound om expected value} \[\mathbb{E}\left[\|\mathbf{U}_{W}^{T}\mathbf{S}^{T}\mathbf{S}\mathbf{T}^{\perp}-\mathbf{U}_{W}^{T}\mathbf{T}^{\perp}\|_{F}^{2}\right]=\sum_{i=1}^{m}\frac{\left\|(\mathbf{U}_{W}^{T})_{:,i}\right\|^{2}\left\|\mathbf{T}^{\perp}_{i,:}\right\|^{2}}{s\mathbf{p}_{i}}-\frac{1}{s}\|\mathbf{U}_{W}^{T}\mathbf{T}^{\perp}\|_{F}^{2}\]
        By substituting the value of $\mathbf{p}$, using the value  of $\gamma$ as in Lemma \ref{embedding for W} and the fact that $\mathbf{U}_{\mathbf{W}}^{T}$, we have 
        \begin{align*} \mathbb{E}\left[\|\mathbf{U}_{W}^{T}\mathbf{S}^{T}\mathbf{S}\mathbf{T}^{\perp}-\mathbf{U}_{W}^{T}\mathbf{T}^{\perp}\|_{F}^{2}\right]&\leq \frac{R\gamma}{\beta s }\|\mathbf{T}^{\perp}\|_{F}^{2}
        \end{align*}
        \item Using Markov's inequality, $\mathbf{Pr}[\left\|\mathbf{U}_{W}^{T}\mathbf{S}^{T}\mathbf{S}\mathbf{T}^{\perp}\right\|_{F}^{2}>t]\leq \frac{\mathbb{E}\left[\left\|\mathbf{U}_{\mathbf{W}}^{T}\mathbf{S}^{T}\mathbf{S}\mathbf{T}^{\perp}\right\|_{F}^{2}\right]}{t}$. By substituting the value of $\mathbb{E}\left[\left\|\mathbf{U}_{W}^{T}\mathbf{S}^{T}\mathbf{S}\mathbf{T}^{\perp}\right\|_{F}^{2}\right]$ and applying it for $t = \frac{\epsilon \|\mathbf{T}^{\perp}\|_{F}^{2}}{2}$ we get
        \[\mathbf{Pr}\left[\left\|\mathbf{U}_{W}^{T}\mathbf{S}^{T}\mathbf{S}\mathbf{T}^{\perp}\right\|_{F}^{2}>\frac{\epsilon \|\mathbf{T}^{\perp}\|_{F}^{2}}{2}\right]\leq \frac{2R\gamma}{\beta \epsilon s}\]
    \end{itemize}
    The proof finished by substituting the value of $s$ given.
\end{proof}


\begin{theorem}
\label{main theorem}
    Consider the least square problem $\|\mathbf{WF}^{T}-\mathbf{T}\|_{F}^{2}$ with $\mathbf{T} = \mathbf{W^{*}\mathbf{F}^{*}}^{T}$. Let $\mathbf{p}$ be a probability distribution proportional to the leverage scores of $\mathbf{W^{*}}$, $\mathbf{p}_{i}\geq \beta l_{i}(\mathbf{W}^{*})/R$ and $\mathbf{S}\in \mathbb{R}^{s\times m}$ be a scaled i.i.d row sampling matrix according to $\mathbf{p}$. For $\epsilon,\delta>0$, if $s\geq R\gamma/\beta \max\left(\frac{4}{\delta \epsilon},\frac{144\ln(2R/\delta)}{\epsilon_{0}^{2}}\right)$, where $\gamma=\max_{i}\left(\frac{\|(\mathbf{U}_{W})_{i,:}\|_{2}^{2}}{\|(\mathbf{U}_{W^{*}})_{i,:}\|_{2}^{2}}\right) $, then the solution to the sketched least square problem \eqref{sketched_least_square}, $\mathbf{F}_{sk}$, and the optimal solution $\mathbf{F}_{opt}$ satisfies the following bounds
    \begin{equation*}
    \begin{aligned}
        \|\mathbf{WF}_{sk}^{T}-\mathbf{T}\|_{F}^{2}
        \le (1+\epsilon)\|\mathbf{WF}_{opt}^{T}-\mathbf{T}\|_{F}^{2} \\[6pt]
        \|\mathbf{F}_{opt}-\mathbf{F}_{sk}\|_{F}^{2}
        \le \frac{\epsilon\,\|\mathbf{WF}_{opt}^{T}-\mathbf{T}\|_{F}^{2}}{\sigma_{\min}^{2}(\mathbf{W})}
    \end{aligned}
    \end{equation*}
\end{theorem}
\begin{proof}
    Using Lemma \ref{SC1_W}, we have that the first structural condition \eqref{structural_cond_1} holds with probability at least $1-\delta/2$. Similarly, using Lemma \ref{matrix approximation for W}, we have that the second structural condition \eqref{structural_cond_2} holds with probability at least $1-\delta/2$. Then the two conditions \eqref{structural_cond_1} and \eqref{structural_cond_2} both hold for probability at least $1-\delta$. The proof finished by applying Theorem \ref{Sketched_lS_Theorem}
\end{proof}

Theorem \ref{main theorem} gives a guarantee at each ALS iteration, if a sampling matrix was drawn independently according to leverage scores of $\mathbf{W}^{*}$. It is important to note that in the full ALS iterations, these $\epsilon-$errors can accumulate. Theorem \ref{main theorem} also 
shows that the number of samples depends on how much $\mathbf{W}$ is aligned with $\mathbf{W^{*}}$. For completeness, we discuss that during ALS optimization, as the factor matrices converge to their true values, the KRP $\mathbf{W}$ converges to the exact matrix $\mathbf{W}^{*}$. This convergence guarantees that the two associated column spaces become increasingly aligned, leading to a decrease in $\gamma$ between iterations.

\begin{proposition}
\label{krp_convergence}
    Let $\mathbf{F}_{k}^{t}$ denote the factor matrices obtained when solving the alternating least squares problem and $\mathbf{F}_{k}^{*}$ be the exact solution. For any $\epsilon_{o}>0$, assume that $\mathbf{F}_{k}^{t}$ converges to $\mathbf{F}_{k}^{*}$, i.e. there exist $t^{0}_{k}(\epsilon_{0})$ such that $\forall t\geq t^{0}_{k}$, $\left\|\mathbf{F}_{k}^{*}\mathbf{PD}_{k}-\mathbf{F}_{k}^{t}\right\|\leq \frac{\epsilon_{0}}{\left(\sum_{j=1}^{N}\|\mathbf{F}_{j}^{*}\mathbf{PD}_{j}\|\right)}$, with $\mathbf{P}$ and $\mathbf{D}_{k}$ are the permutation and scaling matrices, then the KRP of the factor matrices, $\mathbf{W}_{k}^{t}$, converges to the actual KRP $\mathbf{W}_{k}^{*}\mathbf{PD}$ with $\mathbf{D}=\prod_{j=1, j\neq k}^{N}\mathbf{D}_{j}.$
\end{proposition}
\begin{proof}
    For simplicity we give the proof for the KRP of two factor matrices $\mathbf{F_{1}^{t}}$ and $\mathbf{F_{2}^{t}}$ ($N=2$). For $t\geq \max(t^{0}_{1},t_{2}^{0})$
    \begin{align*}
        \left\|\mathbf{F}_{1}^{t}\odot \mathbf{F}_{2}^{t}-(\mathbf{F}_{1}^{*}\odot\mathbf{F}_{2}^{*})\mathbf{PD}_{1}\mathbf{D}_{2}\right\|&=\left\|\mathbf{F}_{1}^{t}\odot \mathbf{F}_{2}^{t}-\mathbf{F}_{1}^{*}\mathbf{PD}_{1}\odot \mathbf{F}_{2}^{t} +\mathbf{F}_{1}^{*}\mathbf{PD}_{1}\odot \mathbf{F}_{2}^{t}-(\mathbf{F}_{1}^{*}\odot\mathbf{F}_{2}^{*})\mathbf{PD}_{1}\mathbf{D}_{2}\right\|\\
        &\leq \left\|\mathbf{F}_{1}^{t}\odot \mathbf{F}_{2}^{t}-\mathbf{F}_{1}^{*}\mathbf{PD}_{1}\odot \mathbf{F}_{2}^{t}\right\|+\left\|\mathbf{F}_{1}^{*}\mathbf{PD}_{1}\odot \mathbf{F}_{2}^{t}-(\mathbf{F}_{1}^{*}\odot \mathbf{F}_{2}^{*})\mathbf{PD}_{1}\mathbf{D}_{2}\right\|\\
        &\leq\left\|\left(\mathbf{F}_{1}^{t}-\mathbf{F}_{1}^{*}\mathbf{PD}_{1}\right)\odot\mathbf{F}_{2}^{t}\right\|+\left\|\mathbf{F}_{1}^{*}\mathbf{PD}_{1}\odot\left(\mathbf{F}_{2}^{t}-\mathbf{F}_{2}^{*}\mathbf{PD}_{2}\right)\right\|\\
        &\leq \|\mathbf{F}_{1}^{t}-\mathbf{F}_{1}^{*}\mathbf{PD}_{1}\|\|\mathbf{F}_{2}^{t}\|+\|\mathbf{F}_{2}^{t}-\mathbf{F}_{2}^{*}\mathbf{PD}_{2}\|\|\mathbf{F}_{1}^{*}\mathbf{D}_{1}\| \\
        &\leq \frac{\epsilon_{0}}{\left(\|\mathbf{F}_{2}^{*}\mathbf{PD}_{2}\|+\|\mathbf{F}_{1}^{*}\mathbf{PD}_{1}\|\right)}\left(\|\mathbf{F}_{2}^{*}\mathbf{PD}_{2}\|+\|\mathbf{F}_{1}^{*}\mathbf{PD}_{1}\|\right)+\frac{\epsilon_{0}^{2}}{\left(\|\mathbf{F}_{2}^{*}\mathbf{PD}_{2}\|+\|\mathbf{F}_{1}^{*}\mathbf{PD}_{1}\|\right)^{2}}\leq  2\epsilon_{0}
    \end{align*}
 This shows convergence.
\end{proof}

\begin{proposition}
\label{columns_space convergence}
    Let $\epsilon_{0}>0$ and $\mathbf{W}_{k}^{t}$ denote the KRP matrix obtained when solving the alternating least squares problem and $\mathbf{W}_{k}^{*}$ be the exact KRP. Assume that there exists $t_{0}$  such that $\forall t\geq t_{0}$, $\|\mathbf{W}_{k}^{*}\mathbf{PD}_{k}-\mathbf{W}_{k}^{t}\|<\epsilon_{0}$, with $\mathbf{P}$ and $\mathbf{D}$ are the permutation and scaling matrices, then $\|(\mathbf{I}-\mathbf{P}_{\mathbf{W}_{k}^{t}})\mathbf{P}_{\mathbf{W}_{k}^{*}}\|\leq \frac{\epsilon_{0}}{\mathbf{D}_{min}\sigma_{\min}(\mathbf{W}_{k}^{*})}$ where $\mathbf{P}_{\mathbf{W}_{k}^{t}}$ and $\mathbf{P}_{\mathbf{W}_{k}^{*}}$ are the projection matrix onto the column space of $\mathbf{W}_{k}^{t}$ and $\mathbf{W}_{k}^{*}$ respectively $\forall t\geq t_{0}$.
\end{proposition}
\begin{proof}
    Let $\mathbf{v} = \frac{\mathbf{W}_{k}^{*}\mathbf{PD}\mathbf{\omega}}{\|\mathbf{W}_{k}^{*}\mathbf{PD}\omega\|}$ be a normalized vector in the range of $\mathbf{P}_{\mathbf{W}_{k}^{*}}$ for $\omega \in \mathbb{R}^{R}$ then we have the following
    \begin{align*}
    \left\|\left(\mathbf{I}-\mathbf{P}_{\mathbf{W}_{k}^{t}}\right)\mathbf{v}\right\|&=\frac{\|\left(\mathbf{I}-\mathbf{P}_{\mathbf{W}_{k}^{t}}\right)\mathbf{W}_{k}^{*}\mathbf{PD}\mathbf{\omega}\|}{\|\mathbf{W}_{k}^{*}\mathbf{PD\omega\|}}=\frac{\|\left(\mathbf{I}-\mathbf{P}_{\mathbf{W}_{k}^{t}}\right)\left(\mathbf{W}_{k}^{*}\mathbf{DP}\omega-\mathbf{W}_{k}^{t}\omega\right)\|}{\|\mathbf{W}_{k}^{*}\mathbf{PD}\omega\|}\\
    &\leq \frac{\|\left(\mathbf{I}-\mathbf{P}_{\mathbf{W}_{k}^{t}}\right)\left(\mathbf{W}_{k}^{*}\mathbf{PD\omega}-\mathbf{W}_{k}^{t}\omega\right)\|}{\mathbf{D}_{\min}\sigma_{\min}(\mathbf{W}_{k}^{*})\|\mathbf{\omega}\|}\leq \frac{\|\mathbf{I}-\mathbf{P}_{\mathbf{W}_{k}^{t}}\|\|\mathbf{W}_{k}^{t}-\mathbf{W}_{k}^{*}\mathbf{P}\mathbf{D}_{k}\|}{\mathbf{D}_{\min}\sigma_{\min}(\mathbf{W}_{k}^{*})}\\
    &\leq \frac{\|\mathbf{W}_{k}^{t}-\mathbf{W}_{k}^{*}\mathbf{P}\mathbf{D}_{k}\|}{\mathbf{D}_{\min}\sigma_{\min}(\mathbf{W}_{k}^{*})}
    \end{align*}
    As this is true for any vector $\mathbf{v}$ in the range of $\mathbf{P}_{\mathbf{W}_{k}^{*}}$ then for $t\geq t_{0}$, $\|(\mathbf{I}-\mathbf{P}_{\mathbf{F}_{k}^{t}})\mathbf{P}_{\mathbf{W}_{k}}^{*}\|\leq \frac{\epsilon_{0}}{{\mathbf{D}_{\min}\sigma_{\min}(\mathbf{W}_{k}^{*})}}$
\end{proof}

Using a similar approach, as in proposition \ref{columns_space convergence} with
$\mathbf{W}_{k}^{t}$ and $\mathbf{W}_{k}^{*}$ interchanged, we get
\begin{equation}
    \|\mathbf{P}_{\mathbf{W}_{k}^{t}}(\mathbf{I}-\mathbf{P}_{\mathbf{W}_{k}^{*}})\| 
    \leq \frac{\|\mathbf{W}_{k}^{t}-\mathbf{W}_{k}^{*}\mathbf{PD}\|}{\sigma_{\min}(\mathbf{W}_{k}^{t})}\leq \frac{\epsilon_{0}}{\sigma_{min}(\mathbf{W}_{k}^{t})},
\end{equation}
Combining both bounds with the triangle inequality
\begin{align*}
    |\ell_{i}(\mathbf{W}_{k}^{*})-\ell_{i}(\mathbf{W}_{k}^{t})| &\leq \|\mathbf{P}_{\mathbf{W}_{k}^{t}}-\mathbf{P}_{\mathbf{W}_{k}^{*}}\| 
    \\
   & \leq \|\mathbf{P}_{\mathbf{W}_{k}^{t}}(\mathbf{I}-\mathbf{P}_{\mathbf{W}_{k}^{*}})\| 
    + \|(\mathbf{I}-\mathbf{P}_{\mathbf{W}_{k}^{t}})\mathbf{P}_{\mathbf{W}_{k}^{*}}\| 
    \leq \epsilon_{0}\left(\frac{1}{\mathbf{D}_{min}\sigma_{min}\left(\mathbf{W}_{k}^{*}\right)}+\frac{1}{\sigma_{min}\left(\mathbf{W}_{k}^{t}\right)}\right),
\end{align*}

If we start with an initialization close to the exact solution or an incoherent initialization, the value of $\gamma$ will be close to 1; hence, the number of samples needed is proportional to $R\ln(R)$. However, if a coherent matrix is used for initialization, an extra number of samples is required for the first few iterations. Subsequently, the number of samples decreases as the factor matrices converge toward the ground-truth, as described in proposition \ref{columns_space convergence}.
\subsection{Theoretical analysis for Hybrid Sampling}

The hybrid sampling technique which was also used by Larsen et al. in \cite{larsen2022practical} to sample the least square problem, and they showed that it  outperforms the standard sampling technique.
The sampling matrix is defined as follows, $\mathbf{S}_{H}\in \mathbb{R}^{s\times m}$
\[\mathbf{S_{H}}=\begin{bmatrix}
    \bf{S_{D}} & \\
     & \bf{S_{R}}
\end{bmatrix}\]
where $\mathbf{S_{D}}\in \mathbb{R}^{\alpha\times \alpha}$ is a deterministic sampling matrix with one nonzero entry per row corresponding to the $\alpha$ sampled indices $\mathbf{d_s}$, and $\mathbf{S_{R}}\in \mathbb{R}^{s_{R}\times (m-\alpha)}$ is a scaled leverage score sampling matrix for the remaining indices $\bf{r}$. 
Theoretical guarantees for this hybrid sampling were provided in \cite{hayashi2025randomized}, showing that if the $\alpha$ rows were selected deterministically according to the approximate leverage scores of $\mathbf{W}$, then taking additional $\xi \phi$ random samples according to the remaining leverages scores ensures the two structural conditions [\eqref{structural_cond_1},\eqref{structural_cond_2}] where $\xi=R-\sum_{i\in \bf{d_s}}l_{i}(\bf{W})$ and  $\phi=\max\left(\frac{C}{\beta}\ln(2R/\delta),\frac{4}{\beta \delta \epsilon}\right)$ with $C = \frac{144}{(1-\sqrt{2})^{2}}$.
By leveraging these results in our analysis, we can conclude that if $\alpha$ rows were selected deterministically according to the approximate leverage scores of $\bf{W^{*}}$, and an additional $\xi \phi$ rows are sampled according to the remaining leverage scores, the same structural conditions are satisfied for $\xi=R-\sum_{i\in \bf{d_{s}}}l_{i}(\bf{W^{*}})$ and  $\phi=\gamma\max\left(\frac{C}{\beta}\ln(2R/\delta),\frac{4}{\beta\epsilon\delta}\right)$ with $C = \frac{144}{\epsilon_{0}^{2}}$ .
To put everything together in our method, we first note that, according to the discussion in Section~\ref{QR_sampling_section}, if \( R \leq n_k \), then the pivots obtained by SE-QRCS on \( \mathbf{T}_k^{T} \) are close to those obtained deterministically via the highest leverage scores of \( \mathbf{W}^{*}_{k} \).
Second, we assume that after selecting the \( \alpha_k \) indices using SE-QRCS (where \( \alpha_k \) is defined in line 6 of Algorithm~\ref{Algorithm1}), the remaining rows are likely incoherent and have small leverage scores. Therefore, we sample the remaining \( s_R \) rows uniformly at random.
Finally, in our experiments, we use the same sampling matrix for a fixed \( k \) across ALS iterations, as this approach performed well empirically and we found no need to resample at every iteration.

%% file: Computational_Experiments.tex
\section{Numerical Experiments}
\label{sec:comp}
The numerical experiments in this section make use of the ITensorCPD.jl\cite{ITensorCPD.jl} library developed in the Julia programming language.
The ITensorCPD.jl library is a flexible library developed to extend the automatic einsum features of the ITensors.jl\cite{ITensor} library to support tensor networks with hyper-rank structure such as the CPD.
The library supports the decomposition of tensors and tensor networks on CPU and GPU and we are currently working to develop methods to support homogeneous/heterogeneous distributed CPU/GPU architectures. 
The computational experiments were run on a Mac M1 Max equipped with 8 high-throughput CPU cores each with a 3.2GHz processing speed and 64 GB of unified RAM. In all the experiments, the sparse embedding used in the SE-QRCS step has a sketching dimension $\lceil3n_{k}\log(n_{k})\rceil$ and sparsity parameter of $\lceil\log(n_{k})\rceil$. Scripts to reproduce the numerical results are available at \url{https://github.com/Israafkh/ITensorCPD-Experiments.git}

\subsection{Synthetic Data} 
In this work, we consider synthetic tensors constructed using two different methods. 
For our first experiment, we construct the synthetic tensor using the method laid out by  Battaglino et. al.\cite{battaglino2018practical}.
With this method, we create a tensor based on known randomly generated weights $\lambda$ and factor matrices $A^{(n)}$. 
We impose a certain collinearity C on the factor matrices, meaning that each pair of column vectors from the same factor matrices satisfies
\[C = \frac{(a_{r}^{(n)})^{T}a_{s}^{(n)}}{\|a_{r}^{(n)}\|\|a_{S}^{(n)}\|}\]
The weight vector $\lambda$ is generated randomly in $[0.2,0.8]$. After construction of the tensor $\mathcal{T_{\text{true}}}$ using the resulting factor matrices and weights, a noise tensor $\mathcal{N}$, with entries drawn from the normal distribution, is added to get the final tensor
\[\mathcal{T}=\mathcal{T_{\text{true}}}+\nu \frac{\|\mathcal{T_{\text{true}}}\|}{\|\mathcal{N}\|}\mathcal{N}\]
with $\nu$ the noise parameter.

For this first experiment, we create a synthetic order $3$ tensor $\mathcal{T}$ with dimension $400$ along each mode. The collinearity and noise parameters are chosen to be $\{0.5,0.9\}$ and $\{0.01,0.1\}$ respectively. The actual and chosen ranks are 5. In the SE-QRCS procedure, a rank of 5 is used in the sketched matrix resulting in a column space of about 700. In the SE-QRCS-based ALS procedure $80$ columns are sampled and the algorithm terminates when the change in the angle between two CPD approximated tensors is less than $10^{-5}$ or the number of iterations exceeds $200$.
The angle between two CPD approximated tensors is computed efficiently by leveraging the Khatri-Rao product structure of the CPD approximation.
This convergence criteria is chosen because the method does not require the exact or approximate norm of the target tensor, $\mathcal{T}$.
After the ALS procedure is complete we compute the fit between the CPD approximated tensor and the target tensor.

\begin{table}[H]
\centering
\begin{tabular}{|c|c|c|c|c|c|}
\hline
  Tests & $C$ & $\nu$ & Fit, ALS & Fit, SE-QRCS-ALS & Columns in $A_{\text{sub}}$ \\ \hline
1  & 0.5 & 0.01  & 0.9899 & 0.9888 & 693    \\ \hline
2  & 0.9 & 0.01  & 0.9894 & 0.9667 & 671  \\ \hline
3  & 0.5 & 0.1   & 0.9005 & 0.8963 & 689  \\ \hline
4  & 0.9 & 0.1   & 0.9003 & 0.8965 & 712  \\ \hline
\end{tabular}
\caption{Results for the Synthetic Tensor. The number of Columns in  $A_{\text{sub}}$ is the mean over the three dimensions.}
\label{ }
\end{table}
In our second experiment, we attempt to compare the 
SE-QRCS and leverage score-based sampled CPD-ALS algorithms.
To compare these methods, we create a synthetic order-3 tensor with dimension $90\times90\times90$.
We construct this synthetic tensor by generating a random rank $100$ CPD approximation where factors are chosen from a uniform distribution between $[-1,1]$.
We recognize that constructing a random tensor in this way introduces little variance in the leverage scores of the tensor.
Therefore, to make the problem more difficult/realistic, the synthetic tensor is modified by amplifying the norm of the first $40$ columns in the mode-1 matricization of the tensor.
These columns are then randomly permuted to ensure that the scores are not coincidentally discovered by their clustered ordering.
\cref{fig:modified synthetic} shows decomposition of the synthetic tensor to three different ranks $90,100, \text{ and } 110$.
One can see from these plots the failure of the leverage score sampling method. 
Because the factor matrices are approximately square,  the leverage scores of each factor matrix provides no discernible information about which columns of the KRP are of importance.
In this case the leverage score method is effectively uniformly sampling the KRP.
To make up for this deficiency, the leverage score sampling method must significantly oversample the problem.
However, the SE-QRCS method is able to find the most important columns of the KRP and, as a result, we are able to decompose the target tensor with significantly fewer sampled rows.

\begin{figure}[H]
    \centering
    \begin{subfigure}[b]{0.32\textwidth}
        \includegraphics[width=\textwidth]{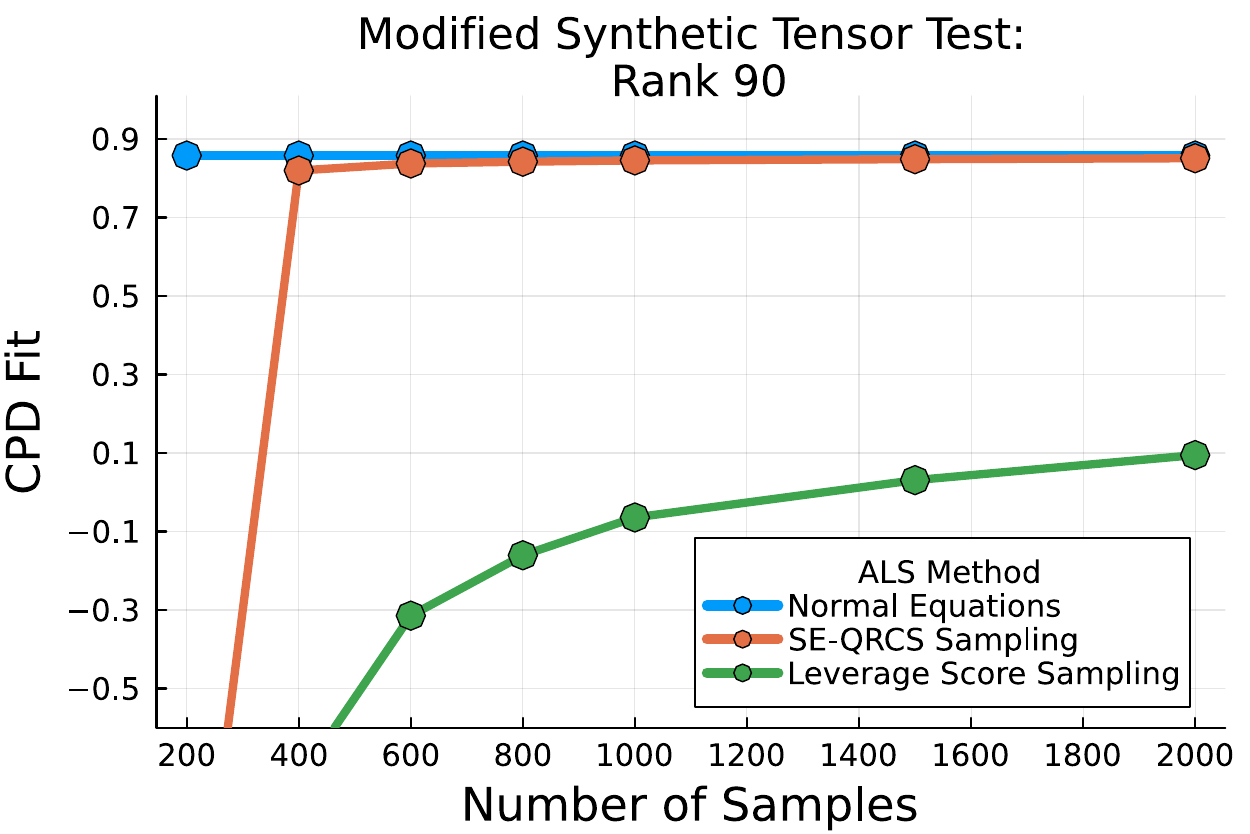}
        \caption{CPD with rank $90$}
    \end{subfigure}
    \hfill
    \begin{subfigure}[b]{0.32\textwidth}
        \includegraphics[width=\textwidth]{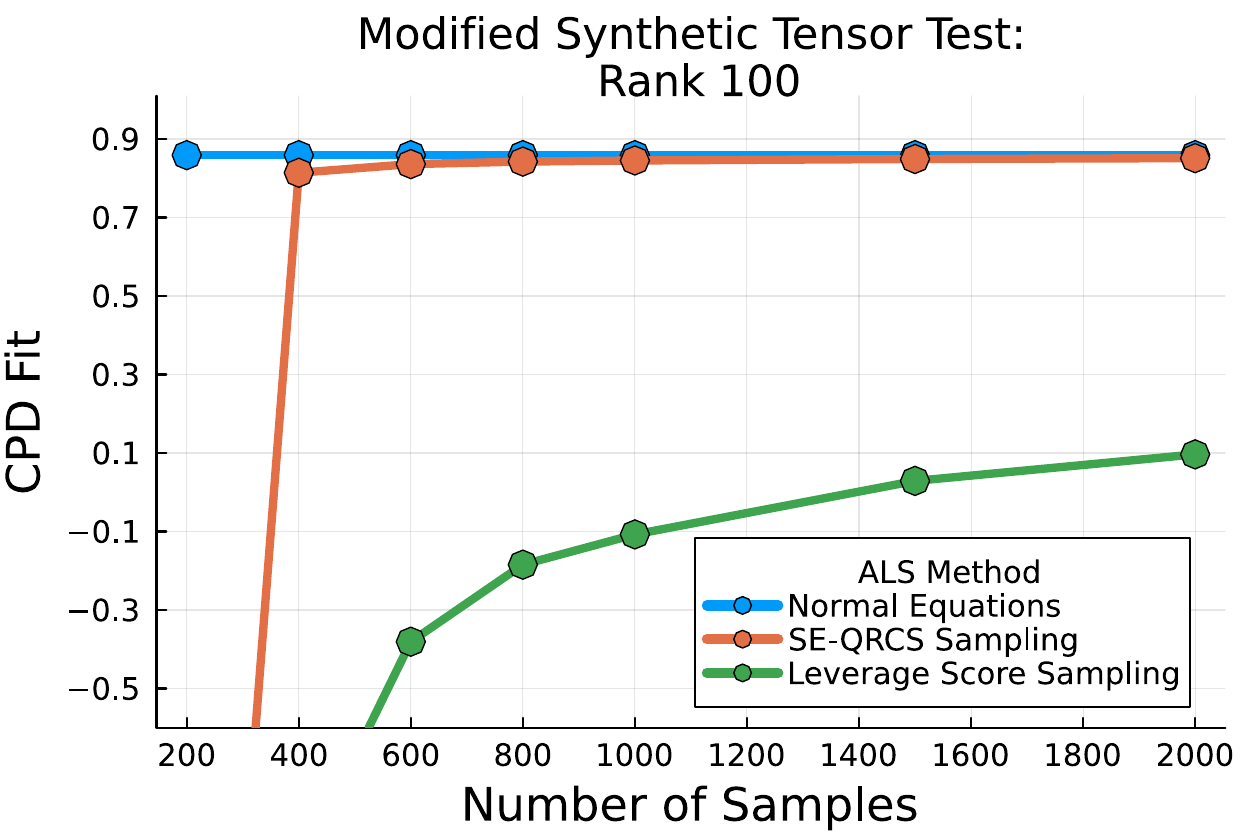}
        \caption{CPD  with rank $100$}
    \end{subfigure}
    \hfill
    \begin{subfigure}[b]{0.32\textwidth}
        \includegraphics[width=\textwidth]{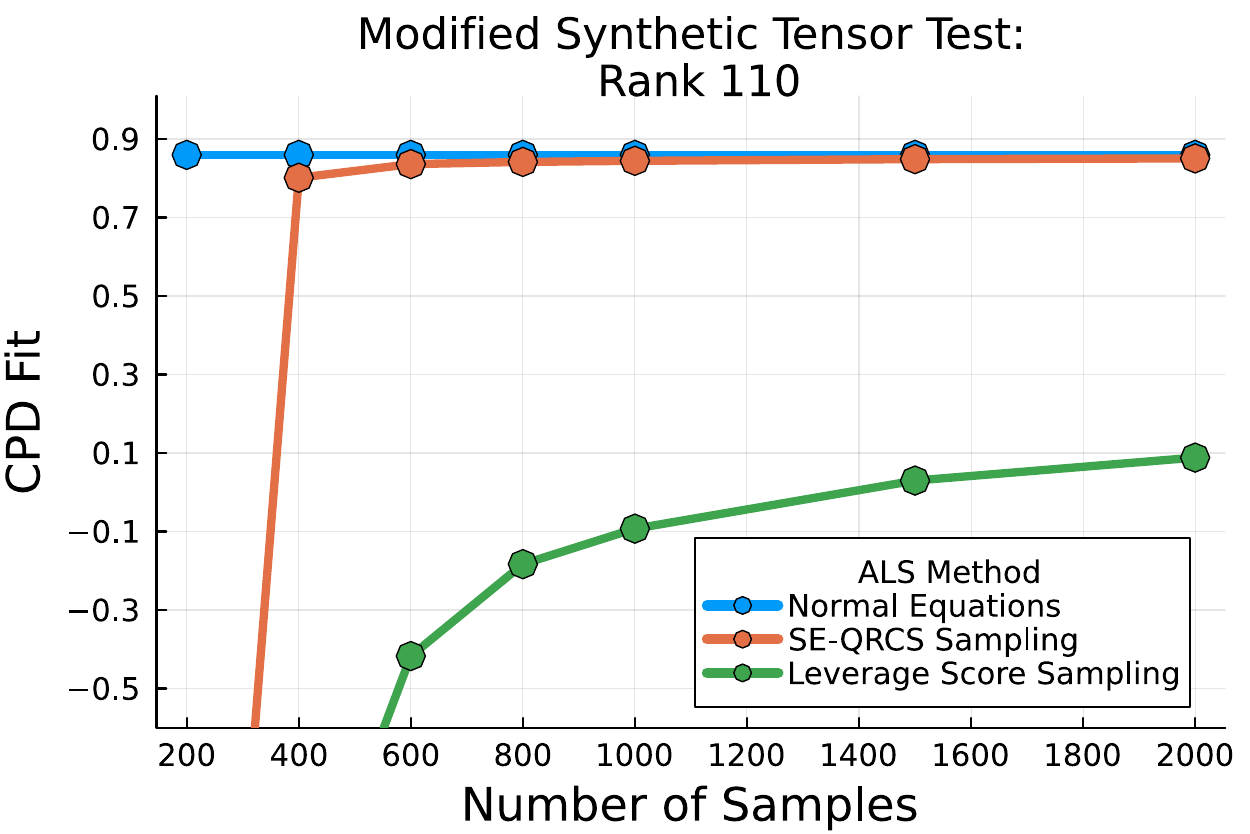}
        \caption{CPD with rank $110$}
    \end{subfigure}
    \caption{Results for modified synthetic tensor}
    \label{fig:modified synthetic}
\end{figure}

\subsection{Real Data} 
We are particularly interested in leveraging this method to decompose physically motivated higher-order tensors.
One of interest is the three-center, two-electron coulomb integral tensor, an important tensor for simulations in quantum chemistry and physics.
Tensor elements can be represented as
\begin{align}\label{eq:b}
    B_{pqX} = \iint \phi^*_{p}(r_1) \phi_q(r_1) g(r_1, r_2)  \psi_X(r_2) dr_1 dr_2
\end{align}
where $g(r_1, r_2)$ is the positive definite kernel $|r_1 - r_2|^{-1}$ and $\phi, \psi$ are single particle functions comprised of products of gaussian functions.
In \cref{eq:b}, there are two basis of one electron function. 
The first basis is denoted by $\phi$ and is comprised of $I_\mu$ single particle functions. 
The dimension of $I_{\mu} \approx 10 * \text{N}_{\text{elec}}$, where $\text{N}_{\text{elec}}$ is the number of electrons in a given chemical system.
The second basis is denoted by $\psi$ and is comprised of $I_{\text{aux}}$ single particle functions.
The dimension of $I_{aux} \approx 3 * I_{\mu}$.
Therefore, the tensor $\mathcal{B} \in \mathbb{R}^{I_\mu \times I_\mu \times I_{\text{aux}}}$ grows very quickly with chemical system size and becomes infeasible to store.
Investigations into the use of the CPD to reduce the storage complexity of this tensor \cite{VRG:benedikt:2011:JCP,VRG:hohenstein:2012:JCP, Pierce:2025:JCTC:MP2,Pierce:2025:JCTC:CPB} have recently been of interest to chemists and physicists.
These studies typically assume the rank of the tensor $\mathcal{B}$ to be between $I_\text{aux}$ and $5 I_\text{aux}$, depending on the chemical system size and necessary accuracy.
In \cref{fig:B}, we decompose the three-center two-electron integral for the 10-water cluster in the TIP4P geometry\cite{jorgensen:1983:JCP,Wales:1998:CPL} using common Gaussian basis sets (the cc-pVTZ orbital basis set and cc-pVTZ-RI auxiliary basis set \cite{Dunning1989,Kendall1992}) where $I_\mu = 540$ and $I_\text{aux} = 1410$, making the $\mathcal{B}$ tensor about 3.3 GB.
We decompose $\mathcal{B}$ to two different CP ranks $R = \{I_\text{aux},2I_\text{aux}\}$ using the SE-QRCS and leverage score sampled CPD-ALS.
In this study, we forgo the canonical CPD-ALS optimization because of the methods high computational costs in time and memory.
For the SE-QRCS decomposition, 200 columns were sampled from the sketched matrix resulting in a column space of roughly $(85,000$, $85,000$, $16,000)$.
These values correspond to roughly 10 percent of the columns in the first and second mode matricization and 5 percent of the columns in the third mode matricization.
Because the rank of the CPD is large, the leverage score CPD-ALS strategy is, again, effectively a uniform sampler of the KRP for each LS subproblem.
From \cref{fig:B} one can see that, when fewer than a critical number of samples is used (around $4 * 10^4$) the SE-QRCS algorithm outperforms the naive uniform sampling method. 
However, when the number of samples surpasses the threshold the SE-QRCS method is roughly equivalent to the leverage score method.
We believe that this phenomenon is associated with the difficulty of this problem, i.e. the $\mathcal{B}$ tensor has both a relatively large CP rank and a large number of important rows in the KRP.
Here we define large as a number on the order of or greater than the dimension of modes on the target tensor.
Although the methods find a relatively equivalent CPD fit in this regime, it's important to remember that the SE-QRCS method does not require the resampling of the target tensor during the ALS optimization procedure.
This, therefore, leads to computational benefits in performance and storage over the leverage score method.
\begin{figure}[t]
    \centering
    \begin{subfigure}[b]{0.49\textwidth}
        \includegraphics[width=\textwidth]{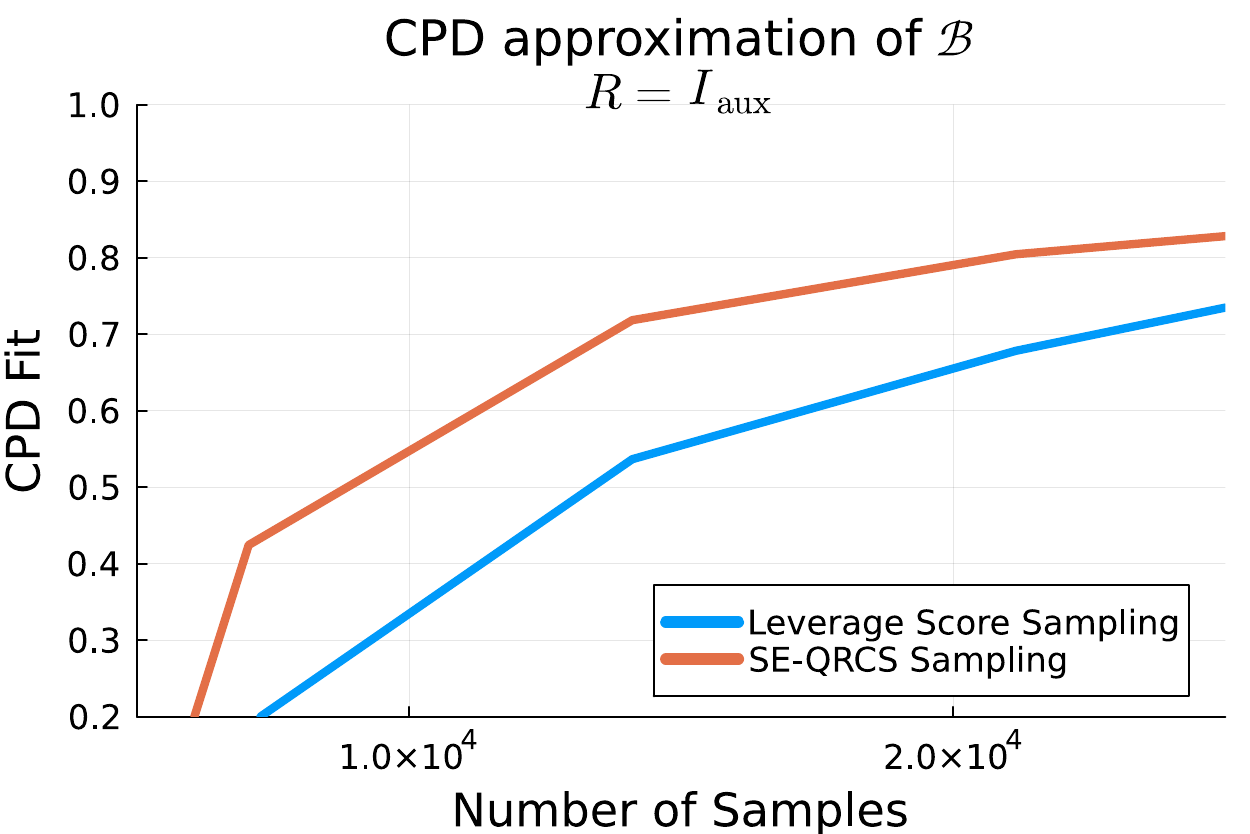}
        \caption{CPD with rank $I_\text{aux}$}
    \end{subfigure}
    \hfill
    \begin{subfigure}[b]{0.49\textwidth}
        \includegraphics[width=\textwidth]{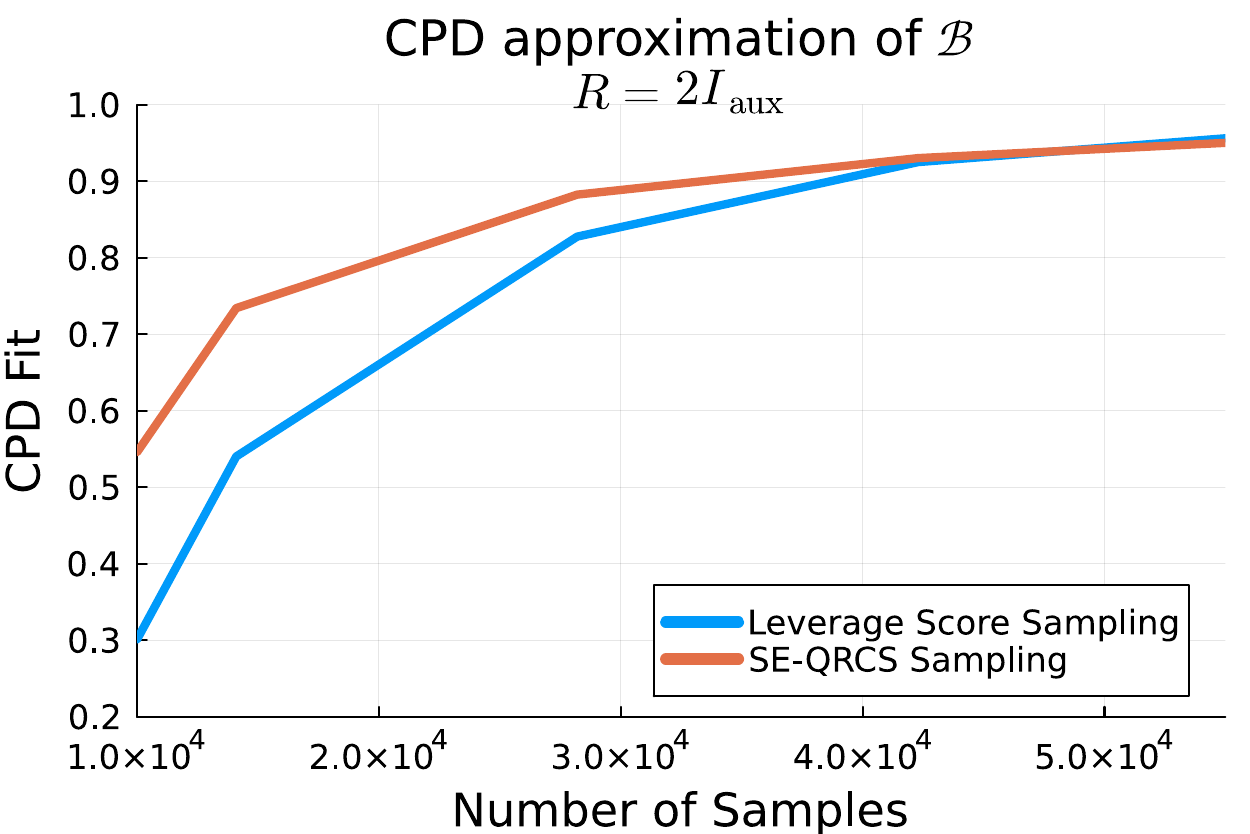}
        \caption{CPD  with rank $2 I_\text{aux}$}
    \end{subfigure}
    \hfill
    \caption{Results for sampled CPD-ALS approximation of the three-center two-electron integral tensor $\mathcal{B}$ for a 10-water cluster using CP rank (a) $R= I_\mathrm{aux}$ and (b) $R= 2 I_\mathrm{aux}$}
    \label{fig:B}
\end{figure}

Next we consider the cost of the SE-QRCS and leverage score sampled CPD-ALS optimizations in \cref{fig:Btime}. 
In this figure, we split the cost of the SE-QRCS methods into the two aforementioned components, the computation of the SE-QRCS of the target tensor and the ALS optimization.
One can see that that the SE-QRCS ALS optimization is consistently faster than the leverage score ALS algorithm.
Given the similarity of these two methods, the overhead of the leverage score method must come from, primarily, resampling the target tensor. 
One can see from \cref{fig:Btime} that the cost associated with computing the SE-QRCS is non-trivial.
However, it is important to note that the SE-QRCS decomposition must only be computed a single time, does not depend on the CPD and does not need to be recomputed if either CPD changes or the number of samples is changed.
To efficiently resample the target tensor, we cache the position of the sorted pivots determined from the SE-QRCS algorithm (holding only $n_k$ for each LS subproblem) and uniformly sample from the column positions not included in this list.
\begin{figure}[t]
    \centering
    \begin{subfigure}[b]{0.49\textwidth}
        \includegraphics[width=\textwidth]{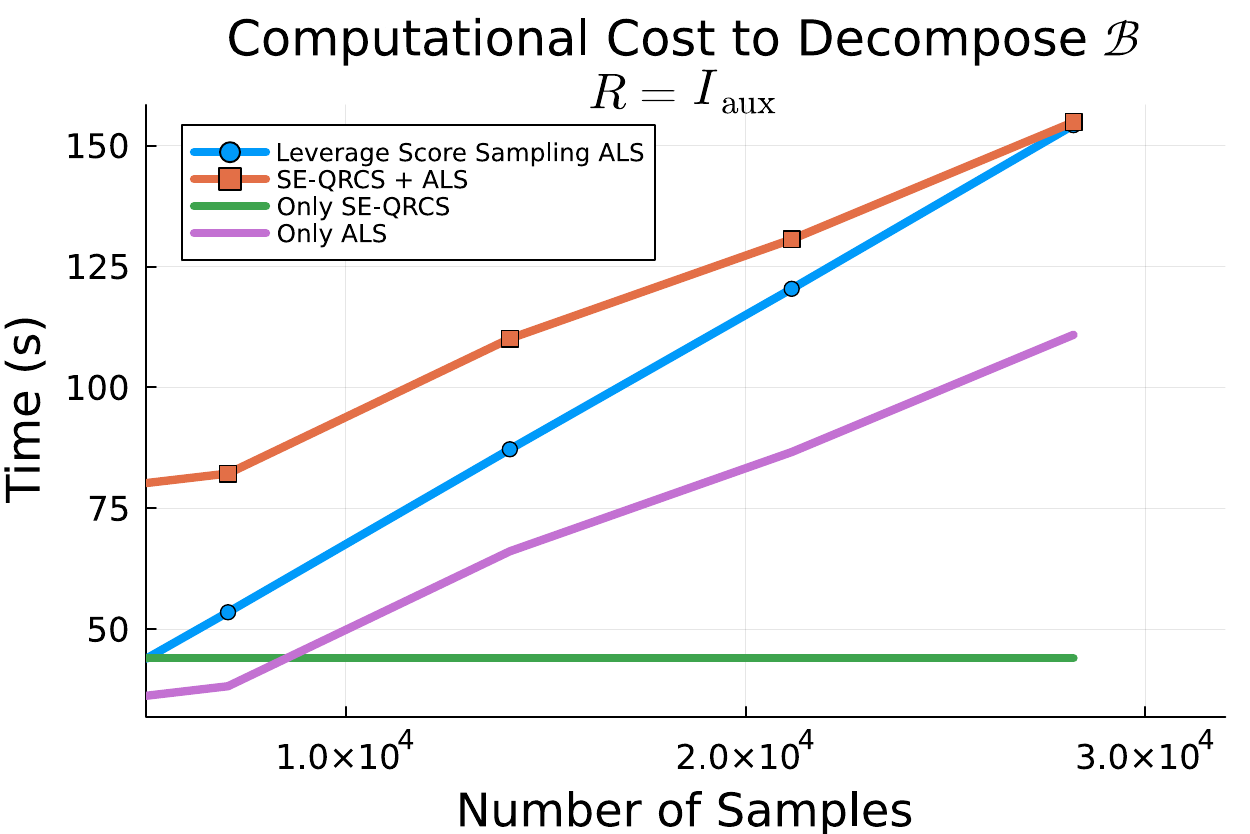}
        \caption{CPD with rank $I_\text{aux}$}
        \label{fig:btime1}
    \end{subfigure}
    \hfill
    \begin{subfigure}[b]{0.49\textwidth}
        \includegraphics[width=\textwidth]{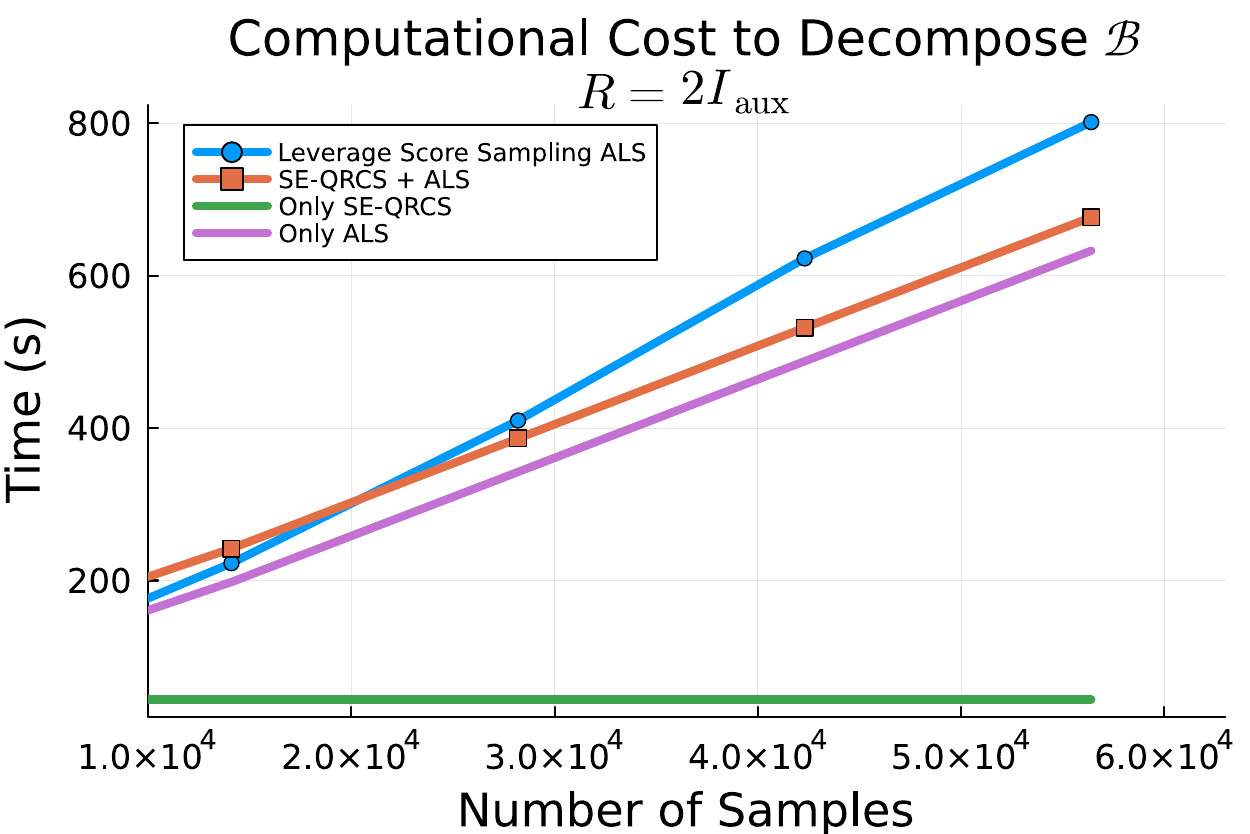}
        \caption{CPD with rank $2I_\text{aux}$}
        \label{fig:btime3}
    \end{subfigure}
    \caption{Wall-clock timing data for the SE-QRCS and leverage score Sampled ALS optimization procedure. (a) decomposes the order three tensor $\mathcal{B}$ is decomposed to a rank $I_\mathrm{aux}$. (b) decomposes $\mathcal{B}$ to a rank of 2$I_\mathrm{aux}$.}
    \label{fig:Btime}
\end{figure}

In \cref{fig:btime1}, i.e. the low-rank regime, when considering the cost of both the SE-QRCS and the CPD-ALS, the method is more expensive than the leverage score-based CPD-ALS method. 
However, one can see that the cost of the leverage score-based CPD-ALS outpaces the cost of the SE-QRCS based CPD-ALS when the number of samples increases.
In \cref{fig:btime3} we study the cost of the decomposition when the rank is increased to $2 I_\text{aux}$.
The target tensor sampling algorithm implicitly depends on the CP rank because the number of necessary samples is proportional to the rank for both the SE-QRCS and leverage score-based CPD-ALS algorithms.
This result demonstrates that, because the SE-QRCS method does not require resampling from the target tensor, the SE-QRCS-based CPD-ALS is faster than the leverage score-based method. 


%% file: Conclusion.tex
\section{Conclusion and Future}
In this work we present a (randomized) QR-based method to determine leverage score-based samples for the sampled CPD-ALS optimization.
By taking advantage of the relationship of the sRRQR and leverages scores to DPP and volume sampling, we are able to efficiently determine statistically important columns of the true (unknown) KRP for each subproblem of the ALS.
The goal of these QR algorithms is to determine the best set of leverage scores for each ALS subproblem before starting the ALS procedure.
Therefore, this work minimizes the number of times the target tensor must be accessed and effectively removes the requirement of storing the full target tensor in the optimization of the CP factor matrices.
Practically, computing the QR of the target tensor has a cost that scales exponentially with the tensor order.
However, because the matricization of a higher-order tensor along a single mode generates a rectangular matrix we may take advantage of randomized linear algebra strategies.
We use the SE-QRCS method to efficiently compute a randomized column-wise pivoted QR.
and show that the number of samples required to create a $(1\pm\epsilon)$ embedding of the ALS least squares problem is $s=\frac{\gamma R}{\beta} \max\left(\frac{4}{\delta \epsilon},\frac{144\ln(2R/\delta)}{\epsilon_{0}^{2}}\right)$, where $\gamma$ captures the coherence between the Khatri-Rao product (KRP) and the matricized tensor.
Furthermore, we illustrate that our SE-QRCS-based ALS algorithm performs equally well or better than the current state of the art, leverage score-based CPD-ALS algorithm.

We understand that there are still a couple of limitations to the SE-QRCS CPD-ALS method.
The first limitation is in the ability of the pivoted QR methods to accurately sort the positions of {\it all} large leverage scores in the correct, descending, order.
Because the QR is a greedy algorithm, we find that for very large tensors and for tensors with large CP rank (i.e. when CP rank is equal to or greater than the dimension of the larger tensor mode) the QR can fail to find the position of some large leverage scores.
As a means to capture these missing values, in this work we take advantage of a hybrid, deterministic/random sampling scheme.
However, we are currently investigating methods to improve the deterministic sampling procedure to find all large valued leverage scores of the KRP.
Our second limitation is related to the current memory requirements of the SE-QRCS algorithm. 
At this time, the SE-QRCS decomposition requires the full target tensor be stored in memory.
As this is a serious limitation for large and high-dimensional data, we are working on a matrix-free implementation to efficiently decompose tensors that either do not fit into disk memory or require large amounts of computing resources to generate elements.
We expect to introduce the necessary modifications to address this memory limitation in near future releases of the ITensorCPD.jl library.
Furthermore, we plan to extend this work to the CPD-ALS of large, high-dimensional tensors, streaming tensors and high-dimensional functions found in physical modeling applications.